\documentclass[a4paper,12pt,oneside,dvipsnames,reqno]{amsart} 

\usepackage{amsmath,amsfonts,amsthm,amssymb}

\usepackage{xcolor}
\usepackage{graphicx}
\usepackage{cite}

\usepackage{bbm}
\usepackage{tikz-cd}

\usepackage{hyperref}
\hypersetup{colorlinks=true, citecolor=PineGreen, linkcolor=RoyalBlue, linktoc=page,urlcolor=RoyalBlue}

\usepackage{geometry}

\theoremstyle{plain}
\newtheorem{theorem}{Theorem}[section] 
\newtheorem{lemmy}[theorem]{Lemma}
\newtheorem{prop}[theorem]{Proposition}

\theoremstyle{remark}
\newtheorem{rem}[theorem]{Remark}

\theoremstyle{definition}
\newtheorem{defn}{Definition}[section]

\newcommand{\A}{\mathbb{A}}

\newcommand{\Af}{\mathbb{A}_{\textrm{f}}}

\newcommand{\R}{\mathbb{R}}

\newcommand{\C}{\mathbb{C}}

\newcommand{\Hb}{\mathbb{H}}

\newcommand{\Z}{\mathbb{Z}}

\newcommand{\Q}{\mathbb{Q}}

\newcommand{\SL}{\operatorname{SL}}

\newcommand{\N}{\mathbb{N}}

\newcommand{\Vol}{\operatorname{Vol}}

\newcommand{\diag}{\operatorname{diag}}

\title{Unipotent mixing for general moduli}

\author{Edgar Assing}

\date{\today}

\email{assing@math.uni-bonn.de}

\begin{document}

\thanks{The author is supported by the Germany Excellence Strategy grant EXC-2047/1-390685813 and also partially funded by the Deutsche Forschungsgemeinschaft (DFG, German Research Foundation) – Project-ID 491392403 – TRR 358.}
\subjclass[2020]{Primary:  11F03, 11N36, 37A17}
\keywords{equidistribution, mixing, horocycles, shifted convolution sums}

\maketitle

\begin{abstract}
In this note we prove a joint equidistribution result for discrete low lying horocycles. This generalizes previous work of Blomer and Michel, where it was crucially assumed that the number of discrete points is prime. 
\end{abstract}


\section{Introduction}

The main result of this paper, given in Theorem~\ref{th:main} below, roughly states the following. As long as there is no \textit{obvious} obstruction, the points 
\begin{equation}
	\left\{ \left(\frac{a+i}{q},\frac{b_qa+i}{q}\right)\colon a=1,\ldots, q \right\} \nonumber
\end{equation}
equidistribute in $\textrm{SL}_2(\Z)\backslash \Hb \times \textrm{SL}_2(\Z)\backslash \Hb$ for $q\to \infty$. This will be shown without any restriction on $q$, generalizing the result in \cite{BM} where $q$ was assumed to be prime. Before giving a precise statement of the main theorem we will make some general introductory remarks.

\subsection{Low-lying horocycles}

Studying the distribution of continuous low lying horocycles in the modular curve is a classical discipline going back to work of Zagier; see \cite{Za}. More precisely, we consider the following commutative diagram.
\begin{center}
\begin{tikzcd}
	\mathbb{R} \arrow[rr, hook] \arrow[d]  &  & \{ x+\frac{i}{T}\colon x\in \mathbb{R}\} \arrow[r, hook] \arrow[d]  &  \mathbb{H}\arrow[d]  \\
	\mathbb{R}/\mathbb{Z} \arrow[rr, hook] &  & \{x+\frac{i}{T}\colon x\in \mathbb{R}/\mathbb{Z}\} \arrow[r, hook] &\textrm{SL}_2(\mathbb{Z})\backslash \mathbb{H}
\end{tikzcd}
\end{center}
It turns out that the so obtained closed horocycle $\{x+i/T\colon x\in \R/\Z\}$ equidistributes as $T\to \infty$. In other words, for $f\in \mathcal{C}_c^{\infty}(\textrm{SL}_2(\mathbb{Z})\backslash \mathbb{H})$, we have
\begin{equation}
	\lim_{T\to \infty}\int_0^1 f(x+i/T)dx = \frac{3}{\pi}\int_{\textrm{SL}_2(\mathbb{Z})\backslash \mathbb{H}} f(x+iy)\frac{dxdy}{y^2}. \nonumber
\end{equation} 
This has been studied in many variations and much stronger results are available. More discussion can for example be found in \cite{Sa, He, St}.

Recall that the hyperbolic metric on $\mathbb{H}$ is given by $ds = \frac{\vert dz\vert}{dy}$, so that the closed horocycle  $\{x+i/T\colon x\in \R/\Z\}$ has length $T$. It is a natural to consider $\asymp T$ well spaced points on this closed horocycle and ask whether the so obtained set still equidistributes. This can be made precise as follows. For an integer $q\in \N$ we consider
\begin{equation}
	\overline{H}_q=\left\{ \frac{a+i}{q}\colon a\text{ mod }q \right\} \subseteq \textrm{SL}_2(\Z)\backslash \mathbb{H}. \label{eq:class_low_lying_hjor}
\end{equation} 
The fact that this set equidistributes as $q\to\infty$ is folklore and we refer to \cite{He1, Ve, BSY} for more discussion. We also record the following phase space version
\begin{equation}
	\lim_{q\to \infty}\frac{1}{q} \sum_{a\text{ mod }q}f\left(\frac{1}{\sqrt{q}}\left(\begin{matrix} 1 & a \\ 0 & q \end{matrix}\right)\right) = \frac{1}{\Vol(\textrm{SL}_2(\mathbb{Z})\backslash \textrm{SL}_2(\mathbb{\R}))}\int_{\textrm{SL}_2(\mathbb{Z})\backslash \textrm{SL}_2(\mathbb{\R})} f(g)dg, \label{eq:fist_phase_equidist_stat}
\end{equation}
for $f\in \mathcal{C}_c^{\infty}(\textrm{SL}_2(\mathbb{Z})\backslash \textrm{SL}_2(\mathbb{\R}))$. Note that in view of our applications it is important to work on the group level and not only on $\textrm{SL}_2(\Z)\backslash\Hb$. We refer to Theorem~\ref{th:S-arithmetic_horocycle} below for a more general statement.

\begin{rem}
Recall that the classical Hecke correspondence is given by 
\begin{equation}
	\mathcal{T}_q = \left\{ \textrm{SL}_2(\Z)\cdot \frac{1}{\sqrt{q}}\left(\begin{matrix} a & b\\ 0 & d\end{matrix}\right)\colon ad=q,\, b\text{ mod }d  \right\} \subseteq \textrm{SL}_2(\mathbb{Z})\backslash \textrm{SL}_2(\mathbb{\R}). \nonumber
\end{equation}
We directly see that, for $q$ prime, the discrete horocycle agrees with $\mathcal{T}_q$ up to the point $\textrm{diag}(\sqrt{q},1/\sqrt{q})$. Since equidistribution for Hecke sets is well known, we directly inherit equidistribution of the discrete horocycle in this situation. This was pointed out in \cite[Section~3.1]{GM}.

If $q$ is not prime, then the difference between the discrete horocycle and the Hecke correspondence is more significant. This is best seen in the extreme example $q=2^n$, where $\frac{1}{2}\sharp \mathcal{T}_q \sim q$ for large $n$.
\end{rem}

An interesting variation of the problem discussed so far is to consider the slightly sparser sets
\begin{equation}
	\overline{H}_q^{\textrm{pr}}=\left\{ \frac{a+i}{q}\colon a\text{ mod }q,\, (a,q)=1 \right\} \subseteq \textrm{SL}_2(\Z)\backslash \mathbb{H}. \label{eq:class_low_lying_prim}
\end{equation} 
Note that $\sharp\overline{H}_q^{\textrm{pr}} = \varphi(q)$, where $\varphi$ is the Euler $\varphi$-function. Equidistribution of these sets (and more) was discussed in \cite{BSY, Ja, ELS} for example. Indeed, by say \cite[Theorem~1]{Ja}, we have  
\begin{equation}
	\lim_{q\to \infty}\frac{1}{\varphi(q)} \sum_{\substack{a\text{ mod }q \\ (a,q)=1}}f\left(\frac{1}{\sqrt{q}}\left(\begin{matrix} 1 & a \\ 0 & q \end{matrix}\right)\right) = \frac{1}{\Vol(\textrm{SL}_2(\mathbb{Z})\backslash \textrm{SL}_2(\mathbb{\R}))}\int_{\textrm{SL}_2(\mathbb{Z})\backslash \textrm{SL}_2(\mathbb{\R})} f(g)dg, \label{eq:fist_phase_equidist_stat_prim}
\end{equation}
for $f\in \mathcal{C}_c^{\infty}(\textrm{SL}_2(\mathbb{Z})\backslash \textrm{SL}_2(\mathbb{\R}))$. An $S$-arithmetic version of this is stated in Theorem~\ref{th:S-arithmetic_horocycle_prim} below.

One can go even further by putting stronger restrictions on $a\text{ mod }q$ obtaining even sparser equidistribution results. Such problems have been considered in \cite{Ve, SU}.

\begin{rem}
Note that we are considering points of a fixed spacing on a periodic orbit of the horocycle flow. It is the ambient orbit that is varying. It is also natural to consider a fixed generic orbit of the horocycle flow and enquire about the distribution of certain points on this. In this situation recent progress has been made in \cite{ELS, FKR}.
\end{rem}

\subsection{Simultaneous equidistribution results}

After having studied the distribution of individual low lying horocycles, we can look at finer properties thereof. A natural next step is to look for mixing and joint equidistribution. In the case of continuous horocylces (purely) dynamical methods are very powerful. We state a special case \cite[Theorem~10.2]{Ma1} adapted to our notation.

\begin{theorem}\label{th:cont_joint_dist}
Let $I\subseteq \R$ be a fixed non-empty interval and let $y\in \R$ be irrational. Then
\begin{equation}
	\left\{ \left(\frac{1}{\sqrt{T}}\left(\begin{matrix} 1 & x \\ 0 & T\end{matrix}\right),\frac{1}{\sqrt{T}} \left(\begin{matrix} 1 & yx \\ 0 & T\end{matrix}\right) \right)\colon x\in I  \right\} \subseteq {\normalfont \textrm{SL}}_2(\Z)\backslash {\normalfont \textrm{SL}}_2(\R) \times {\normalfont \textrm{SL}}_2(\Z)\backslash {\normalfont \textrm{SL}}_2(\R) \nonumber
\end{equation}	
equidistributes as $T\to\infty$ with respect to the uniform probability measure.
\end{theorem}

\begin{rem}
This result covers \cite[Theorem~1.3]{BM}, where it has to be assumed that $y$ is badly approximable in a weak but quantitative form. Note that the proof of Theorem~\ref{th:cont_joint_dist}, as presented in \cite{Ma1}, reduces to an application of Shah's theorem. This argument was recently rediscovered in \cite{Bu}.
\end{rem}

Note that the condition $y\in \R\setminus\Q$ in Theorem~\ref{th:cont_joint_dist} is optimal. Indeed, this translates into the property that the lattices $\textrm{SL}_2(\Z)$ and $\diag(1,y)^{-1}\textrm{SL}_2(\Z) \diag(1,y)$ are not commensurable. However, if $y\in \Q$, then, as a result of the commensurability of the two lattices, the flow gets trapped in a Hecke correspondence. This feature is used in the proof of \cite[Theorem~5.(B)]{Ma2} and is also explained in \cite{Bu}.

The main result of this paper concerns a discrete version of Theorem~\ref{th:cont_joint_dist}. To state this result we have to find a suitable replacement for the condition $y\in \R\backslash \Q$. To do so we introduce the lattices
\begin{equation}
	\Lambda_{q,b} = \{ (n_1,n_2)\in \Z^2\colon n_1+bn_2\equiv 0\text{ mod }q \}\nonumber
\end{equation}
and define $s(q;b)$ as the first successive minimum of $\Lambda_{q,b}$. See Section~\ref{sec:lattice} for more details and some useful properties of $\Lambda_{q,b}$ and $s(q;b)$. In particular, we refer to Remark~\ref{rem:p_adic_s} below which contains a discussion on how $s(q;b)$ can relate to rationality. We are now ready to state our main theorem.

\begin{theorem}\label{th:main}
Let $((q_n,b_n))_{n\in \N}$ be a sequence of integers with $(q_n,b_n)=1$ and $s(q_n;b_n)\to\infty$. Then
\begin{equation}
	\overline{H}_{q_n,b_n} = \left\{ \left(\frac{a+i}{q_n},\frac{ab_n+i}{q_n}\right)\colon a \text{ mod }q_n \right\} \nonumber
\end{equation}
equidistributes in ${\normalfont \textrm{SL}}_2(\Z)\backslash \Hb \times {\normalfont \textrm{SL}}_2(\Z)\backslash \Hb$ with respect to the uniform probability measure.
\end{theorem}

\begin{rem}
An extension of Theorem~\ref{th:main} to more then one factor along the lines of \cite[Corollary~1.2]{BM} follows from \cite[Corollary~1.5]{EL} as soon as $(q_n,37\cdot 53)=1$ for all $n\in \N$. We expect that this condition can be removed with a little more work, but have not attempted to carry this out here.
\end{rem}

This was proven for sequences where $q_n$ is prime in \cite[Theorem~1.1]{BM}. Once one allows general moduli, it is natural to wonder whether it is possible to restrict $a$ to $(a,q_n)=1$. In this direction we prove the following result.

\begin{theorem}\label{th:main_prim}
There are absolute constants $A,C>0$ so that the following holds. Let $((q_n,b_n))_{n\in \N}$ be a sequence of integers with $(q_n,b_n)=1$ and $s(q_n;b_n)\to\infty$. Fix $\epsilon>0$ and assume that 
\begin{itemize}
	\item $2^{\omega(q_n)}\ll \log(q_n)^{\frac{1}{10}}$ and $q_n^{\epsilon}\ll s(q_n;b_n)$;
\end{itemize}
or
\begin{itemize}
	\item $\min(2^{(2+\epsilon)\omega(q_n)}, \exp((C+2)\log\log(q_n)\log\log\log(q_n))) \ll s(q_n;b_n) \ll q_n^{\frac{1/2-2\theta}{A}-\epsilon}$,
\end{itemize}
where $0\leq \theta<\frac{1}{4}$ is the currently best known bound towards the Ramanujan-Petersson conjecture.\footnote{In view of \cite[Appendix 2, Proposition~2]{Ki} one can take $\theta = \frac{7}{64}$.} Then 
\begin{equation}
	\overline{H}_{q_n,b_n}^{\textrm{pr}} = \left\{ \left(\frac{a+i}{q_n},\frac{ab_n+i}{q_n}\right)\colon a \in \left(\Z/q_n\Z\right)^{\times} \right\} \nonumber
\end{equation}
equidistributes in ${\normalfont \textrm{SL}}_2(\Z)\backslash \Hb \times {\normalfont \textrm{SL}}_2(\Z)\backslash \Hb$ with respect to the uniform probability measure.
\end{theorem}

Here $\omega(q_n)$ denotes the number of prime divisors of $q_n$. The technical conditions on the growth of $s(q_n;b_n)$ or $2^{\omega(q_n)}$ in Theorem~\ref{th:main_prim} arise in our treatment of certain Weyl sum. It seems reasonable to believe that the statement holds without them. However, at this point it is unclear to us how one could remove the conditions completely.

\subsection{The methods}

Let us start by briefly explaining the argument from \cite{BM}, which beautifully combines analytic estimates for Weyl sums with the joinings theorem of Einsiedler and Lindenstrauss. Let us describe very briefly describe the idea, for a more detailed summary we refer to \cite[Section~1.2]{BM}. Instead of considering the measures supported on $\overline{H}_{q_n,b_n}$ directly one lifts them to measures $\mu^S_{q_n,b_n}$ on a suitable $S$-arithmetic quotient, where the places in $S$ are co-prime to all $q_n$. In this setting the measure classification theorem of Einsiedler and Lindenstrauss becomes applicable. As a result one can write a limit points of the sequence $(\mu_{q_n,b_n}^S)_{n\in \N}$ as a convex combination of the uniform measure and a collection of diagonal measures. To exclude the diagonal contribution one tests the measures against carefully constructed test-functions. This leads to a classical Weyl sums, which are estimated using tools from analytic number theory. Ultimately one shows that the Weyl sums become arbitrary small as $n\to \infty$ and this rules out the possibility of mass on shifted diagonals.

While the part of the argument that relies on measure classification goes through as soon as there are two fixed primes $l_1, l_2$ with $(q_n,l_1l_2)=1$, one has to work harder to show the corresponding estimates for Weyl sums. Nonetheless, one can establish the following result.

\begin{theorem}\label{th_shallow}
Let $(q_n,b_n)$ be a sequence of pairs such that $s(b_n;q_n)\to\infty$ as $n\to \infty$ and let $N_0\mid (37\cdot 53)^{\infty}$ be fixed. Furthermore, assume that $(q_n,37\cdot 53)=1$ for all $n\in \N$. Then we have
\begin{multline}
	\lim_{n\to\infty} \frac{1}{q_n}\sum_{a\text{ mod } q_n} f\left(\frac{1}{\sqrt{q_n}}\left(\begin{matrix}1 & a\\0&q_n\end{matrix}\right),\frac{1}{\sqrt{q_n}}\left(\begin{matrix}1 & b_na\\0&q_n\end{matrix}\right)\right) \\
	= \frac{1}{\Vol(\Gamma_1(N_0)\backslash \SL_2(\R))^2}\int_{\Gamma_1(N_0)\backslash \SL_2(\R)}\int_{\Gamma_1(N_0)\backslash \SL_2(\R)} f(g_1,g_2)dg_1dg_2, \nonumber
\end{multline} 
for all $f\in \mathcal{C}_c^{\infty}(\Gamma_1(N_0)\backslash \SL_2(\R) \times \Gamma_1(N_0)\backslash \SL_2(\R))$.
\end{theorem}

This covers already a wide range of (sequences of) moduli $q_n$. However, for example the sequence
\begin{equation}
	q_n = p_1\cdots p_n, \nonumber
\end{equation}
where $p_k$ is the $k$th prime number can not be treated directly using this method. This can be resolved with a trick. Indeed, we can free up the places $37$ and $53$ using the Chinese Remainder Theorem and apply equidistribution of the remaining sequence. To execute this strategy it turns out to be crucial that we have the freedom of choosing a (fixed) congruence subgroup $\Gamma_1(N_0)$ in Theorem~\ref{th_shallow} and that the result therein is stated for (quotients of) $\SL_2(\R)$ and for $\Hb$. This works as long as the exponents of $37$ and $53$ stay bounded in the sequence $q_n$.

Even with this trick we can not cover all moduli $q_n$. Indeed, the sequence
\begin{equation}
	q_n=n! \nonumber
\end{equation}
can not be handled using these techniques. The reason is that we are not able to free up two fixed primes and can therefore not hope to apply the joinings theorem. The key to handling this case is the observation that as $n\to\infty$ arbitrary large powers of $37$ (or of $53$) divide the moduli. This allows us to establish invariance by a unipotent subgroup and to apply Ratner's measure classification result directly. This can be combined with an appropriate Weyl sum estimate as explained earlier and leads to the following result.

\begin{theorem}\label{th_depth}
Let $(q_n,b_n)$ be a sequence of pairs with $(q_n,b_n)=1$ and $s(q_n;b_n)\to\infty$ as $n\to \infty$. Suppose that there is $l\in \{37,53\}$ such that $v_l(q_n)\to\infty$ as $n\to\infty$. Then we have
\begin{multline}
	\lim_{n\to\infty} \frac{1}{q_n}\sum_{a\text{ mod } q_n} f\left(\frac{1}{\sqrt{q_n}}\left(\begin{matrix}1 & a\\0&q_n\end{matrix}\right),\frac{1}{\sqrt{q_n}}\left(\begin{matrix}1 & b_na\\0&q_n\end{matrix}\right)\right) \\ = \frac{9}{\pi^2}\int_{\SL_2(\Z)\backslash \Hb}\int_{\SL_2(\Z)\backslash \Hb} f(g_1,g_2)dg_1dg_2, \nonumber
\end{multline} 
for all $f\in \mathcal{C}_c^{\infty}(\SL_2(\Z)\backslash \Hb \times \SL_2(\Z)\backslash \Hb)$.
\end{theorem}

Together Theorem~\ref{th_shallow} and Theorem~\ref{th_depth} cover all possible sequences $q_n$ and lead to Theorem~\ref{th:main}. Note that all our results are ineffective. The reason for this is on the one hand our use of the joinings theorem of Einsiedler and Lindenstrauss. On the other hand, we also loose effectiveness when distinguishing sequences according to their behaviour at the places $37$ and $53$.

The strategy for proving Theorem~\ref{th:main_prim} is very similar. We first establish a result for sequences that stay away from $37$ and $53$. See Theorem~\ref{th_shallow_primitive} below. This is combined with Theorem~\ref{th_depth_primitive}, which handles the orthogonal situation.

\section{Notation and Set-up}

\subsection{Some remarks on lattices}\label{sec:lattice}

Let $\Lambda$ be a sub-lattice of $\Z^2$. Then we write $s(\Lambda)$ for its minimum. A key role in this paper is played by the specific lattices
\begin{equation}
	\Lambda_{q,b} = \{ (n_1,n_2)\in\mathbb{Z}^2\colon n_1+bn_2 \equiv 0\text{ mod }q \}. \label{eq:def_lattice}
\end{equation}
To shorten notation we write $s(q;b) = s(\Lambda_{q,b})$. 

Note that if we assume $(b,q)= 1$, then $\Lambda_{q, b}$ has the basis
\begin{equation}
	\left(\begin{matrix} -b \\ 1 \end{matrix}\right), \left(\begin{matrix} q \\ 0\end{matrix}\right).
\end{equation}
Thus, in this case the (co)-volume of $\Lambda_{q, b}$ is $q$ and we observe that
\begin{equation}
	0<s(q;b) = s(q;-b)\ll q^{\frac{1}{2}}. \nonumber
\end{equation}

We start by recording some properties of the lattices $\Lambda_{q, b}$. Note that we will explicitly state when the assumption $(b,q)=1$ is used.

\begin{lemmy}\label{lm:prop_reduced_bas}
There is a reduced basis
\begin{equation}
	\left(\begin{matrix} x_1 \\ x_2 \end{matrix}\right), \left(\begin{matrix} y_1 \\ y_2\end{matrix}\right). \nonumber
\end{equation}
of $\Lambda_{q,b}$, were $(x_1,x_2)^{\top}$ is of minimal length. In particular we have $\vert x_1\vert+\vert x_2\vert \ll s(q;b)$. Let $r=(x_1,x_2)$ and write $r=r_1r_2$ with $(r_1,r_2)=(r_1,x_1/r)=(r_2,x_2/r)=1$, then we can reformulate the condition $(n_1,n_2) \in \Lambda_{q,b}$ as $n_1x_2/r - n_2x_1/r =h\cdot q/r$ and $n_i\equiv hy_i \text{ mod }r_i$ for $h\in \Z$ and $i=1,2$.
\end{lemmy}
\begin{proof}
The existence of a reduced basis with $\vert x_1\vert + \vert x_2\vert \ll s(q;b)$ follows from the geometry of numbers. See \cite[Lecture~X]{Si}. 

To see the second claim we take $(n_1,n_2)^{\top}\in \Lambda_{q,b}$ and write
\begin{equation}
	\left(\begin{matrix} n_1 \\ n_2 \end{matrix}\right)= r\cdot \left(\begin{matrix} x_1 \\ x_2 \end{matrix}\right) + h\cdot  \left(\begin{matrix} y_1 \\ y_2\end{matrix}\right),\nonumber
\end{equation}
for uniquely determined $r,h\in \Z$. After recalling that $\Lambda_{q,b}$ has co-volume $q$ we obtain
\begin{equation}
	x_2n_1-x_1n_2 = h\cdot(y_1x_2-y_2x_1) = h\cdot q. \nonumber
\end{equation}
Dividing by $r$ gives the first condition. The congruence conditions modulo $r_1$ and $r_2$ respectively follow immediately. It is clear that this argument is reversible.
\end{proof}

\begin{rem}\label{rm:Wieser}
Note that, if the minimal vector $(x_1,x_2)^{\top}$ of $\Lambda_{q,b}$ is primitive, then $(n_1,n_2)\in \Lambda_{q,b}$ can be detected by the rather simple congruence $n_1x_2\equiv n_2x_1 \text{ mod }q$. While for prime $q$ the minimal vector is automatically primitive, this is not the case for general $q$. The following nice example was communicated to us by Andreas Wieser and Zuo Lin. For $n\in \N$ we take $q=3^{4n}$ and $b=1+3^{3n}$, then the minimal vector is $(-3^n,3^n)$. One checks that the smallest primitive vector is much larger.  
\end{rem}

\begin{lemmy}
For $d\mid q$ we have
\begin{equation}
	s(q;b) \leq \frac{q}{d}s(d;b). \label{eq:minimum_bound}
\end{equation}
\end{lemmy}
\begin{proof}
We take a minimal vector $\mathbf{x}=(x_1,x_2)$ for $\Lambda_{d,b}$, so that $s(d;b) = \Vert \mathbf{x}\Vert$. Now we observe that $(\frac{q}{d}x_1,\frac{q}{d}x_2)\in \Lambda_{q,b}$. So that $s(q;b) \leq \frac{q}{d} \Vert \mathbf{x}\Vert = \frac{q}{d}s(d;b)$ as desired.
\end{proof}

Note that \eqref{eq:minimum_bound} is trivial for $d\leq \sqrt{q}$. However, we expect it to be quite good when $d$ is close $q$.

\begin{rem}\label{rm:pass_to_subseq}
The lemma above will be useful in the following situation. Suppose we have a sequence $\{(q_n,b_n)\}_{n\in \N}$ of tuples with $(q_n,b_n)=1$. Further, we assume that we can write $q_n=q_0q_n'$ with $(q_0,q_n')=1$ for all $n\in \mathbb{N}$. Then $s(q_n;b_n)\to \infty$ implies that $s(q_n';b_n)\to\infty$.
\end{rem}

For later reference we also record 
\begin{lemmy}\label{lm:comparison_s}
Let $(b,q)=1$. Then we have
\begin{equation}
	s(q;fb) \asymp_f s(q;d),
\end{equation}
for all $f\in \N$. If $\overline{b}$ denotes the modular inverse of $b$ modulo $q$, then
\begin{equation}
	s(q;b)=s(q;\overline{b}). \nonumber
\end{equation}
\end{lemmy}
\begin{proof}
We first record that, as long as $(f,q)=1$, we can apply \cite[Lemma~4.1]{BM} and get
\begin{equation}
	d^{-1}s(q;b)\leq \min(s(q;bd),s(q;b\overline{d})) \leq \max(s(q;bd),s(q;b\overline{d})) \leq ds(q;b) \label{eq:co-prime_stuff}
\end{equation}
and $s(q;b)=s(q;\overline{b})$.

We turn to the case when $f'=(f,q) \neq 1$. Consider the map 
\begin{equation}
	\Lambda_{q,fb} \ni (n_1,n_2) \mapsto (n_1,fn_2) \in \Lambda_{q,b}.\nonumber
\end{equation}
This implies $s(q;b) \leq f\cdot s(q;fb)$ giving the desired lower bound for $s(q;fb)$. On the other hand we observe that if $(n_1,n_2)\in \Lambda_{q,fb}$, then $f'\mid n_1$. This allows us to define the map
\begin{equation}
	\Lambda_{q,fb}\ni (n_1,n_2) \mapsto (n_1/f',n_2)\in \Lambda_{q/f',f/f'b}. \nonumber
\end{equation}
This implies that
\begin{equation}
	s(q;fb) \leq s(\frac{q}{f'};\frac{fb}{f'}) \leq \frac{f}{f'}s(\frac{q}{f'};b).\nonumber
\end{equation}
Note that in the last step we have used \eqref{eq:co-prime_stuff}. By trivially estimatting $s(q/f';b) \leq s(q;b)$ we are done.
\end{proof}

Recall that the main condition of our theorem is that our sequence $(q_n,b_n)$ satisfies $s(q_n;b_n)\to\infty$. If $q_n$ is a sequence of primes, then it was pointed out in \cite[Remark~1.3]{BM} that this condition means that $\frac{b_n}{q_n}$ stays far from rational numbers with small denominators. We end this section with a characterization of this condition in an orthogonal situation.

\begin{rem}\label{rem:p_adic_s}
Let $b\in \Z_p^{\times}$ and consider the sequence $(q_n,b_n)$ with $q_n=p^n$ and $b_n=b\text{ mod }q_n$. We claim that $s(q_n;b_n)\to\infty$ if and only if $b\not \in \Z_p^{\times}\cap \Q$. 

First, suppose $b=\frac{r}{s}\in\Q$ with $(r,s)=1$. Write $\frac{r}{s}=b= b_n+q_nb_n'$. Then we have
\begin{equation*}
	\vert r-b_ns\vert_p \leq \max(\vert r-bs\vert_p,\vert q_nb_n's\vert_p) \leq q_n^{-1}.
\end{equation*}
In particular, we have $(r,-s)\in \Lambda_{q_n,b_n}$ so that $s(q_n;b_n)$ is bounded. 

On the other hand, suppose that $s(q_n;b_n)\leq R$. Not that there are only finitely many tuples of integers in a ball of radius $R$. Thus there is $\mathbf{m}=(m_1,m_1)\in \Z^2$ with $\Vert \mathbf{m}\Vert\leq R$ and a subsequence $(n_i)_{i\in \N}$ such that
\begin{equation*}
	\vert m_1 + b_{n_i}m_2 \vert_p \leq q_{n_i}^{-1}
\end{equation*}
for all $i\in \N$. After excluding situations when $0\in \{m_1,m_2\}$ we can put $t=\frac{-m_1}{m_2}\in \Q^{\times}$ and $r=v_p(m_2)$. We have
\begin{equation*}
	\vert t-b\vert_p \leq \max(q_{n_i}^{-1},\vert t-b_{n_i}\vert_p) = \max(q_{n_i}^{-1},\vert m_1+b_{n_i}m_2\vert_p\vert m_2\vert_p^{-1}) \leq p^{r-n_i},
\end{equation*}
for all $i\in \N$. We conclude that $b=t$ is rational as desired.
\end{rem}

\subsection{The $S$-arithmetic setting}

We now recall some notation from \cite[Section~5.1]{BM}, which allows us to work on $S$-arithmetic quotients. 

Given a place $v$ of $\Q$ we write $\Q_v$ for the corresponding completion. If $v=p$ is a finite place, then let $\Z_p\subseteq \Q_p$ denote the closure of $\Z$. Let $\Af$ denote the finite adeles
\begin{equation}
	\Af = {\prod_p}' \Q_p.
\end{equation}
The adeles are then defined as $\A=\R\times \Af$.

Given an integer $D\in \N$ we define the finite set of places
\begin{equation}
	S_D=\{\infty\}\cup \{ p\mid D\}. \label{eq:def_SD}
\end{equation}
In case the integer $D$ is clearly determined by the context we drop it from the notation and simply write $S=S_D$. We introduce the spaces
\begin{equation}
	\Q_S = \prod_{v\in S}\Q_v,\, \Z_S=\prod_{v\in S}\Z_p \text{ and }\Z[1/S]=\Z\left[\prod_{p\in S}\frac{1}{p}\right]. \nonumber
\end{equation}
On the other hand we have
\begin{equation}
	\A^{(S)}={\prod_{v\not\in S}}'\Q_v \text{ and }\Z^{(S)}=\prod_{p\not\in S}\Z_p. \nonumber
\end{equation}

To shorten notation we write $G=\textrm{PGL}_2$. We fix the maximal compact subgroup $K_{\infty}=\textrm{PSO}_2(\R) \subseteq G(\R)$. At finite place $v=p$ we write $K_p=G(\Z_p)$ and consider the congruence subgroup
\begin{equation}
	K_p(N) = \left\{ \left(\begin{matrix} a & b \\ c & d\end{matrix}\right)\in K_p\colon c,d-1\in N\Z_p \right\}\subseteq K_p. \nonumber
\end{equation}
Finally, given a place $v$ and a matrix $g_v\in G(\Q_v)$ we write $$\iota_v(g_v) = (1,\ldots, 1,g_v,1,\ldots)\in G(\A).$$

Globally we will work with the compact open subgroup
\begin{equation*}
	K(N)^{(S)} = \prod_{p\not\in S} K_p(N) \subseteq G(\A^{(S)}).
\end{equation*}
Recall that strong approximation for $\SL_2$ together with surjectivity of $\det\colon K(N)^{(S)}\to \Z^{(S),\times}$ yields
\begin{equation}
	G(\A) = G(\Q) G(\Q_S)K(N)^{(S)}.\nonumber
\end{equation}
In particular, we obtain the (natural) identification
\begin{equation}
	X(N)^S := G(\Q)\backslash G(\A)/K(N)^{(S)} \cong \Gamma_1(N)_S\backslash G(\Q_S), \label{eq:identification_S}
\end{equation}
where
\begin{equation}
	\Gamma_1(N)_S = G(\Q)\cap K(N)^{(S)} \subseteq G(\Z[1/S]). \nonumber
\end{equation}
As a concrete example we can take $S=\{\infty\}$, then 
\begin{equation}
	X(N)^{\infty} =\Gamma_1(N)_{\infty}\backslash G(\R) = \Gamma_1(N)\backslash \SL_2(\R), \nonumber
\end{equation}
for the classical congruence subgroup
\begin{equation}
	\Gamma_1(N) = \left[ \begin{matrix} \Z & \Z \\ N\Z & 1+NZ\end{matrix}\right] \cap \SL_2(\Z).\nonumber
\end{equation}
Note that this notation is a bit peculier in the sense that $X(N)^S = X(1)^S$ when the prime divisors of $N$ lie in $S$.

We employ the standard notation $[G(\A)] = G(\Q)\backslash G(\A)$ for the full adelic quotient. Combining the canonical quotient map with the identification in \eqref{eq:identification_S} we obtain the map
\begin{equation}
	\pi_{S,N} \colon [G(\A)] \to X(N)^S.
\end{equation}

We turn towards an $S$-arithmetic description of (discrete) low-lying horocycles. First, we define the map
\begin{align}
	i_q\colon \prod_{p\mid q} \Z_p &\to [G(\A)] \nonumber\\
	(x_p)_{p\mid q} &\mapsto G(\Q)\iota_{\infty}\left(\left(\begin{matrix} 1 & 0\\ 0 & q\end{matrix}\right)\right)\prod_{p\mid q}\iota_p\left(\left(\begin{matrix} 1 & \frac{-x_p}{q} \\ 0 & 1\end{matrix}\right)\right). \nonumber
\end{align}  
Given integers $D,N$ we further define $i_q^{(S_D,N)} = \pi_{S_D,N}\circ i_q$. The image of $i_q^{(S_D,N)}$ carries a natural measure denoted by $\mu_{q}^{(S_D,N)}$ coming from the Haar probability measures on $\Z_p$ (for $p\mid q$).

\begin{lemmy}\label{lm:co_prime_shape_sets}
Suppose that $(D,q)=1$. The image of $i_q^{S_D,N}$ is given by
\begin{equation}
	H_q^{S_D,N} = \{x_{q,a}^{S_D}\colon a\in \Z/q\Z\}, \nonumber
\end{equation}
where
\begin{equation}
	x_{q,a}^{S_D} = \Gamma(N)_{S_D}\left(u_{a/q}\left(\begin{matrix} 1 & 0\\ 0 & q\end{matrix}\right), u_{a/q},\ldots,u_{a/q}\right)\in X(N)^S \text{ and }u_{a/q}=\left(\begin{matrix} 1 & a/q \\ 0&1\end{matrix}\right). \nonumber
\end{equation}
Furthermore, we have
\begin{equation}
	\mu_q^{(S_D,N)} = \frac{1}{q}\sum_{a\text{ mod } q} \delta_{x_{q,a}^{S_D}}. \nonumber
\end{equation}
\end{lemmy}
\begin{proof}
This follows directly after noting that, for $(D,q)=1$, the map $i_q^{(S_D,N)}$ factors through $\prod_{p\mid q} \Z_p/q\Z_p \cong \Z/q\Z$.
\end{proof}

\begin{rem}
Note that if $D=N=1$, then we have
\begin{equation}
	x_{q,a}^{\{\infty\}} = \textrm{PGL}_2(\Z)\cdot \left(\begin{matrix}
		1 & a \\ 0 & q\end{matrix}\right) \in \textrm{PGL}_2(\Z)\backslash \textrm{PGL}_2(\R). \nonumber
\end{equation}
After identifying $\textrm{PGL}_2(\Z)\backslash \textrm{PGL}_2(\R)$ and $\SL_2(\Z)\backslash \SL_2(\R)$ we note that $\mu_q^{(\{\infty\},1)}$ is precisely the measure appearing on the left hand side of \eqref{eq:fist_phase_equidist_stat}. 

Going one step further we can also quotient by $K_{\infty}$ on the right. Using the usual identification
\begin{equation}
	\textrm{PGL}_2(\Z)gK_{\infty} \mapsto g.i 
\end{equation}
we observe that $H_q^{\{\infty\},1}$ projects to $\overline{H}_q$ defined in \eqref{eq:class_low_lying_hjor}. 
\end{rem}

We can now state the following $S$-arithmetic version of the equidistribution result from \eqref{eq:fist_phase_equidist_stat}.

\begin{theorem}\label{th:S-arithmetic_horocycle}
For fixed $D,N\in \N$ we have 
\begin{equation}
	\lim_{q\to \infty} \mu_q^{S_D,N}(f) = \frac{1}{\Vol(X(N)^{S_D})}\int_{X(N)^{S_D}}f(x) dx, \nonumber
\end{equation}
for $f\in \mathcal{C}_c^{\infty}(X(N)^{S_D})$.
\end{theorem}

Similarly we can consider the $S$-arithmetic analogue of \eqref{eq:class_low_lying_prim}. Indeed, here we simply consider the restriction $i_q^{\textrm{pr}}$ of $i_q$ to $\prod_{p\mid q}\Z_p^{\times}$. For integers $D,N$ we then obtain the map
\begin{equation}
	\tilde{i}_q^{(S_D,N)}\colon \prod_{p\mid q}\Z_p^{\times} \overset{i_q^{\textrm{pr}}}{\to} [G(\A)] \overset{\pi_{S_D,N}}{\to} X(N)^{S_D}. \nonumber 
\end{equation}
The space $\prod_{p\mid q}\Z_p^{\times}$ is equipped with the Haar probability measure and we denote its image under $\tilde{i}_q^{(S_D,N)}$ by $\tilde{\mu}_q^{(S,N)}$. We directly obtain the following variation of Lemma~\ref{lm:co_prime_shape_sets}.

\begin{lemmy}\label{lm:co_prime_shape_sets_prim}
Suppose that $(D,q)=1$. The image of $\tilde{i}_q^{S_D,N}$ is given by
\begin{equation}
	\widetilde{H}_q^{S_D,N} = \{x_{q,a}^{S_D}\colon a\in (\Z/q\Z)^{\times}\}. \nonumber
\end{equation}
Furthermore, we have
\begin{equation}
	\tilde{\mu}_q^{(S_D,N)} = \frac{1}{\varphi(q)}\sum_{\substack{a\text{ mod } q \\ (a,q)=1}} \delta_{x_{q,a}^{S_D}}. \nonumber
\end{equation}
\end{lemmy}

We record the following $S$-arithmetic version of \eqref{eq:fist_phase_equidist_stat_prim}.

\begin{theorem}\label{th:S-arithmetic_horocycle_prim}
For fixed $D,N\in \N$ we have 
\begin{equation}
	\lim_{q\to \infty} \tilde{\mu}_q^{S_D,N}(f) = \frac{1}{\Vol(X(N)^{S_D})}\int_{X(N)^{S_D}}f(x) dx, \nonumber
\end{equation}
for $f\in \mathcal{C}_c^{\infty}(X(N)^{S_D})$.
\end{theorem}

\section{Special Weyl sums}

Throughout this section we take $\mathbf{x}_0=(x_0,x_0')\in \R^2$, $\textbf{y}_0=(y_0,y_0')\in (\R^{\times})^2$ and $r_0\in \Q$. These numbers will be considered fixed in the sense that we usually do not specify the dependence of our estimates on these parameters. 

Let $q\in \N$, $b\in \Z$ and $l\mid q$. For two one-periodic functions $f_1,f_2\colon \mathbb{H}\to \C$ we define the Weyl sum
\begin{equation}
	\mathcal{W}_{f_1,f_2}(b,l,q;\mathbf{x}_0,\mathbf{y}_0,r_0) = \frac{1}{q}\sum_{a\text{ mod }q} f_1\left(\frac{a+x_0'+y_0'i}{q/l}\right)f_2\left(\frac{ba+x_0+y_0i}{q}+r_0\right). \label{eq:def_Weyl_sum}
\end{equation}

\begin{rem}
This is a slightly generalized version of the Weyl sum defined in \cite[(4.1)]{BM}. Indeed, we have additional parameters $l$, $x_0'$ and $y_0'$. Including these does not make our arguments much more complicated, but it will be very convenient when applying our estimates later on. The main difficulty addressed in this section is obtaining estimates for $\mathcal{W}_{f_1,f_2}(b,l,q;\mathbf{x}_0,\mathbf{y}_0,r_0)$ with no restrictions on $q$ and $b$.
\end{rem}

We will also consider a version where the $a$-sum is restricted to primitive points
\begin{equation}
	\mathcal{W}_{f_1,f_2}^{\textrm{pr}}(b,l,q;\mathbf{x}_0,\mathbf{y}_0,r_0) = \frac{1}{\varphi(q)}\sum_{\substack{a\text{ mod }q \\ (a,q)=1}} f_1\left(\frac{a+x_0'+y_0'i}{q/l}\right)f_2\left(\frac{ba+x_0+y_0i}{q}+r_0\right). \label{eq:def_Weyl_sum_pr}
\end{equation}

In our application $f_1$ and $f_2$ will be certain cuspidal Hecke-Maa\ss\  newforms, which are in particular one-periodic. In this case estimating $\mathcal{W}_{f_1,f_2}(b,l,q;\mathbf{x}_0,\mathbf{y}_0,r)$ reduces to a shifted convolution problem involving the Hecke eigenvalues of $f_1$ and $f_2$. To see this we recall that the Fourier expansion of Maa\ss\ forms from \cite[Theorem~3.1]{Iw}. More precisely, we use the normalization from \cite[(4.2)]{BM} and write $f_j$ as
\begin{equation}
	f_j(z) = \frac{\sqrt{2y}}{L_j}\sum_{n\in \N} \lambda_j(n)K_{it_j}^{\ast}(ny)\cdot \left[e(nx)+\epsilon_j e(-nx) \right], \nonumber
\end{equation}
for $\epsilon_j\in \{\pm 1\}$ and
\begin{equation}
	L_j=L(1,\textrm{sym}^2f_j)^{\frac{1}{2}},\, K_{it}^{\ast}(x) = \cosh(\pi t)^{\frac{1}{2}}K_{it}(2\pi x). \nonumber
\end{equation}
Note that the parameter $t_j$ is (up to sign) determined by the Laplace eigenvalue $\lambda_j=\frac{1}{4}+t_j^2$ of $f_j$. If $f_j$ satisfies the Ramanujan-Petersson conjecture, then $t_j\in \R$ and we have the bounds $\vert \lambda_j(n)\vert \leq \tau(n)$, where $\tau$ is the divisor function.

Inserting this in \eqref{eq:def_Weyl_sum} and executing the $a$-sum leads to
\begin{equation}
	\mathcal{W}_{f_1,f_2}(b,d,q;\mathbf{x}_0,\mathbf{y}_0,r_0) = \frac{2}{q}\sum_{\pm}\sum_{\substack{n_1,n_2\in \N,\\ q\mid ln_1\pm bn_2}}\lambda_1(n_1)\lambda_2(n_2)\sum_{\pm}\epsilon_{\pm,\pm} e(\pm n_2r_0)G\left(\frac{n_1}{q},\pm\frac{n_2}{q}\right), \label{eq:original_Weyl_sum}
\end{equation}
where
\begin{equation}
	G(x_1,x_2) = \frac{2e(lx_0'x_1+x_0x_2)}{L_1L_2} K_{it_1}^{\ast}(\vert ly_0'x_1\vert )K_{it_2}^{\ast}(\vert y_0 x_2\vert)\nonumber
\end{equation}
and 
\begin{equation}
	\epsilon_{+,+}=1,\, \epsilon_{+,-}=\epsilon_1\epsilon_2,\, \epsilon_{-,+}=\epsilon_1 \text{ and } \epsilon_{-,-}=\epsilon_2. \nonumber
\end{equation}

In order to handle the primitive case we recall the following identity for Ramanujan sums
\begin{equation}
	c_q(n) = \sum_{\substack{a\text{ mod }q \\ (a,q)=1}}e(\frac{an}{q}) = \sum_{d\mid (q,n)} \mu(q/d)d = \varphi(q)\mu(q/(q,n))\varphi(q/(q,n))^{-1}.\nonumber
\end{equation}
With this at hand we can expand
\begin{multline}
	\mathcal{W}_{f_1,f_2}^{\textrm{pr}}(b,q;\mathbf{x}_0,\mathbf{y}_0,r_0) = \sum_{d\mid q}\frac{d}{q\varphi(q)}\mu(q/d) \sum_{\pm} \sum_{\substack{n_1,n_2\in \N \\ d\mid ln_1\pm bn_2}} \lambda_1(n_1)\lambda_2(n_2) \\
	\cdot \sum_{\pm } \epsilon_{\pm,\pm}e(\pm n_2 r_0)G(\frac{n_1}{q},\pm \frac{n_2}{q}). \label{eq:original_pr_Weyl_sum}
\end{multline}

We record the estimate
\begin{equation}
	\vert x_1\vert^{j_1}\vert x_2\vert^{j_2}\frac{d^{j_1}}{dx_1^{j_1}}\frac{d^{j_2}}{dx_2^{j_2}}G(x_1,\pm x_2) \ll_{\epsilon,A,j_1,j_2,l,\mathbf{x}_0,\mathbf{y}_0} \vert x_1x_2\vert^{-\epsilon}(1+\vert x_1\vert+\vert x_2\vert)^{-A} \label{eq:est_G}
\end{equation}
from \cite[p. 15]{BM}. In particular, at the cost of a negligible error, we can truncate the $n_1$- and $n_2$-sum in $\mathcal{W}_{f_1,f_2}(b,l,q;\mathbf{x}_0,\mathbf{y}_0,r)$ and $\mathcal{W}_{f_1,f_2}^{\textrm{pr}}(b,l,q;\mathbf{x}_0,\mathbf{y}_0,r)$ at $n_1,n_2\ll q^{1+\epsilon}$.

At this point we can already derive an estimate, which turns out to be useful for $s(q;b)$ not to small.

\begin{lemmy}\label{lm:sieve_est}
Suppose $f_1$  and $f_2$ are $L^2$-normalized Hecke-Maa\ss\ newforms for which the Ramanujan-Petersson conjecture holds at all places. Then we have
\begin{equation}
	\mathcal{W}_{f_1,f_2}(b,l,q;\mathbf{x}_0,\mathbf{y}_0,r_0) \ll_{f_1,f_2,l,\mathbf{x}_0,\mathbf{y}_0,\epsilon} \frac{q^{\epsilon}}{s(q;b)}+\frac{\log\log(q)^{12}}{\log(q)^{\frac{1}{9}}}.\label{eq:first_sieve_bound}
\end{equation}
Furthermore,
\begin{equation}
	\mathcal{W}_{f_1,f_2}^{\textrm{pr}}(b,l,q;\mathbf{x}_0,\mathbf{y}_0,r_0) \ll_{f_1,f_2,l,\mathbf{x}_0,\mathbf{y}_0,\epsilon} \frac{q^{\epsilon}}{s(q;b)}+\frac{\log\log(q)^{13}}{\log(q)^{\frac{1}{9}}}\cdot 2^{\omega(q)},\label{eq:second_sieve_bound}
\end{equation}
where $\omega(q)$ is the number of prime divisors of $q$. Note that the implicit constants in both estimates only depend polynomially on $l$, $\vert x_0'\vert$, $\vert x_0\vert$ and $\vert y_0'\vert+\vert y_0'\vert^{-1}$, $\vert y_0\vert+\vert y_0\vert^{-1}$ and are completely independent of $r_0$.
\end{lemmy}
\begin{proof}
After estimating \eqref{eq:original_Weyl_sum} trivially and handling the tail using \eqref{eq:est_G} we need to estimate
\begin{equation}
	\mathcal{W} = \frac{1}{q}\sum_{\substack{ln_1\pm bn_2\equiv 0\text{ mod }q\\ n_1,n_2\ll q^{1+\epsilon}}} \vert \lambda_1(n_1)\lambda_2(n_2)\vert \left\vert G\left(\frac{n_1}{q},\pm \frac{n_2}{q}\right)\right\vert. \nonumber
\end{equation}	 
	
This will be done using a sieve estimate as in \cite[pp. 16-17]{BM}. We define $z=q^{\gamma}$ for some sufficiently small $\gamma>0$ and consider the annuli
\begin{equation}
	\mathcal{R}_{k,d_1,d_2} = \{(x_1,x_2)\in \mathbb{R}_+^2\colon (k-1)q<\sqrt{(d_1x_1)^2+(d_2x_2)^2}\leq kq \}. \nonumber
\end{equation}
Further we temporarily set
\begin{align}
	\Gamma_{d_1,d_2}&=\{ (n_1,n_2)\in \Z^2 \colon d_1n_1\pm b d_2n_2\equiv 0\text{ mod }q\} \text{ and } \nonumber \\
	\Gamma_{d_1,d_2}' &= \{ (n_1,n_2) \in \Z^2\colon d_1\mid n_1, d_2\mid n_2 \text{ and } n_1\pm b n_2\equiv 0\text{ mod }q \}. \nonumber
\end{align}
Define $\mathcal{S}_{k,d_1,d_2}=\Gamma_{d_1,d_2}\cap \mathcal{R}_{k,d_1,d_2}$. As in \cite[Section~3]{BM} we write $\mathcal{C}_q(X,Y)$ for the collection of all subsets $\mathcal{S}\subseteq \mathbb{N}^2$ contained in a ball of radius $X^{100}$ around the origin such that 
\begin{equation}
	\sharp \{\mathbf{n}=(n_1,n_2)\in \mathcal{S}\colon  s_1\mid  n_1 \text{ and } s_2\mid n_2\} = \frac{X}{s_1s_2}+O(Y),\nonumber
\end{equation}
for all $(s_1s_2,q)=1$. Observe that $\sharp \mathcal{S}_{k,d_1,d_2} = \sharp (\mathcal{R}_{k,1,1}\cap \Gamma_{d_1,d_2}')$. We conclude that, for
\begin{multline}
	X_{k,d_1,d_2} = \frac{\Vol(\mathcal{R}_{k,1,1})}{[\Z^2\colon \Gamma_{d_1,d_2}']} \ll \frac{kq((d_1,d_2),q)}{d_1d_2} \\ \text{ and }Y_{k,d_1,d_2}=\frac{{\rm length}(\partial\mathcal{R}_{k,1,1})}{s(\Gamma_{d_1,d_2}')} \ll \frac{kq}{s(q;b)}, \nonumber
\end{multline}
we have $\mathcal{S}_{k,d_1,d_2}\in \mathcal{C}_q(X_{k,d_1,d_2},Y_{k,d_1,d_2})$ as long as $d_1,d_2\mid q^{\infty}$. 
	
Using \eqref{eq:est_G} we estimate
\begin{multline}
	\mathcal{W} \ll_{f_1,f_2,l,\mathbf{x}_0,\mathbf{y}_0,A,\epsilon} q^{-1}\sum_{k\ll q^{\epsilon}}k^{-A}\sum_{\substack{d_1,d_2\mid q^{\infty} \\ d_1,d_2\ll q^{1+\epsilon}, \\ l\mid d_1}} \vert\lambda_1(d_1)\lambda_2(d_2)\vert \\ \cdot \sum_{\substack{(n_1,n_2)\in \mathcal{S}_{k,d_1,d_2}\\ (n_1n_2,q)=1}} \vert \lambda_1(n_1)\lambda_2(n_2)\vert. \nonumber
\end{multline}
Applying \cite[Corollary~3.1]{BM} to estimate the inner sum yields
\begin{multline}
	\mathcal{W} \ll_{f_1,f_2,\mathbf{x}_0,\mathbf{y}_0,A,\epsilon,\gamma} \left(q^{\epsilon-\gamma/4}+\log(q)^{-\frac{1}{9}}\right)\cdot \sum_{\substack{d_1,d_2\mid q^{\infty} \\ d_1,d_2\ll q^{1+\epsilon},\\ l\mid d_1}} \frac{((d_1,d_2),q)}{d_1d_2}\vert \lambda_1(d_1)\lambda_2(d_2)\vert \\ + \frac{q^{\epsilon+3\gamma}}{s(q;b)}\sum_{\substack{d_1,d_2\mid q^{\infty} \\ d_1,d_2\ll q^{1+\epsilon},\\ l\mid d_1}}\vert \lambda_1(d_1)\lambda_2(d_2)\vert. \nonumber
\end{multline}
Using the estimate $\lambda_i(d_i) \leq \tau(d_i)$ we can estimate
\begin{multline}
	\sum_{\substack{d_1,d_2\mid q^{\infty} \\ d_1,d_2\ll q^{1+\epsilon},\\ l\mid d_1}} \frac{((d_1,d_2),q)}{d_1d_2}\vert \lambda_1(d_1)\lambda_2(d_2)\vert \leq \left(\sum_{\substack{d_1,d_2\mid q^{\infty} \\ d_1,d_2\ll q^{1+\epsilon}}} \frac{((d_1,d_2),q)}{d_1d_2}\right)^4 \\
	\leq \left(\sum_{\substack{d\mid q^{\infty} \\ d\ll q^{1+\epsilon}}} \frac{(d,q)}{d^2} \left(\sum_{\substack{t\mid q^{\infty} \\ t\ll q^{1+\epsilon}/d}}\frac{1}{t}\right)^2\right)^4 \leq  \left(\sum_{\substack{t\mid q^{\infty}\\ t\ll q^{1+\epsilon}}}\frac{1}{t}\right)^{12}.\nonumber
\end{multline}
Let $k=\omega(q)$ be the number of prime divisors of $q$. Note that $k\ll \log(q)$ and that by the prime number theorem the $k$th prime is bounded by $k\log(k)$. By Mertens'  theorem we can estimate
\begin{equation}
	\sum_{\substack{t\mid q^{\infty}\\ t\ll q^{1+\epsilon}}}\frac{1}{t} \leq \prod_{p\mid q} (1-\frac{1}{p})^{-1} \leq \prod_{p\ll k\log(k)}(1-\frac{1}{p})^{-1} \ll \log(k) \ll \log\log(q). \nonumber
\end{equation}
We conclude that 
\begin{equation}
	\sum_{\substack{d_1,d_2\mid q^{\infty} \\ d_1,d_2\ll q^{1+\epsilon},\\ l\mid d_1}} \frac{((d_1,d_2),q)}{d_1d_2}\vert \lambda_1(d_1)\lambda_2(d_2)\vert \ll \log\log(q)^{12}. \nonumber
\end{equation}
Using the Rankin-Trick we also see that
\begin{equation}
	\sum_{\substack{d_1,d_2\mid q^{\infty} \\ d_1,d_2\ll q^{1+\epsilon},\\ l\mid d_1}}\vert \lambda_1(d_1)\lambda_2(d_2)\vert \ll q^{\epsilon}. \nonumber
\end{equation}
Choosing $\gamma$ appropriately concludes the proof of \eqref{eq:first_sieve_bound}.

We turn towards estimating $\mathcal{W}_{f_1,f_2}^{\textrm{pr}}(b,l,q;\mathbf{x}_0,\mathbf{y}_0,r_0)$. Similarly, by \eqref{eq:original_pr_Weyl_sum} and \eqref{eq:est_G} it is sufficient to handle
\begin{equation}
	\mathcal{W}^{\textrm{pr}} = \sum_{d\mid q} \frac{d\mu(q/d)^2}{\varphi(q)}\cdot\underbrace{\frac{1}{q}\sum_{\substack{ln_1\pm bn_2\equiv 0\text{ mod }d\\ n_1,n_2\ll q^{1+\epsilon}}} \vert \lambda_1(n_1)\lambda_2(n_2)\vert \left\vert G\left(\frac{n_1}{q},\pm \frac{n_2}{q}\right)\right\vert}_{=\colon\mathcal{W}_d}.\nonumber
\end{equation}
Adapting the argument above allows us to estimate $\mathcal{W}_d$ using \cite[Corollary~3.1]{BM} by
\begin{equation}
	\mathcal{W}_d \ll \frac{q}{d}\cdot \frac{\log\log(q)^{12}}{\log(q)^{\frac{1}{9}}}+\frac{q^{\epsilon}}{s(d;b)}.
\end{equation} 
Inserting this above yields
\begin{equation}
	\mathcal{W}^{\textrm{pr}}\ll \frac{\log\log(q)^{12}}{\log(q)^{\frac{1}{9}}} \cdot \sum_{d\mid q} \frac{d\mu(q/d)^2}{\varphi(q)} \frac{q}{d}  + q^{\epsilon}\sum_{d\mid q} \frac{d\mu(q/d)^2}{\varphi(q)} \frac{1}{s(d;b)}.
\end{equation}
Using the estimate $\frac{1}{\varphi(q)} \ll \frac{1}{q}\log\log(q)$ we obtain 
\begin{equation}
	\frac{\log\log(q)^{12}}{\log(q)^{\frac{1}{9}}} \cdot \sum_{d\mid q} \frac{d\mu(q/d)^2}{\varphi(q)} \frac{q}{d} \ll \frac{\log\log(q)^{13}}{\log(q)^{\frac{1}{9}}} \cdot 2^{\omega(q)}. \nonumber
\end{equation}
On the other hand, using \eqref{eq:minimum_bound} we can easily establish 
\begin{equation}
	q^{\epsilon}\sum_{d\mid q} \frac{d\mu(q/d)^2}{\varphi(q)} \frac{1}{s(d;b)} \ll \frac{q^{\epsilon}}{s(q;b)}. \nonumber
\end{equation}
Combining everything yields \eqref{eq:second_sieve_bound} and completes the proof.
\end{proof}

If $s(q;b)$ is not too big, then another approach is necessary. Here one uses a shifted convolution estimate.

\begin{lemmy}\label{lm:shifted_conv_est}
Suppose $f_1$  and $f_2$ are $L^2$-normalized Hecke-Maa\ss\ newforms of level $N_1$ and $N_2$ for which the Ramanujan-Petersson conjecture holds at all places. Let $r_0\in \Q^{\times}$ and assume that the denominator $d_0$ of $r_0$ satisfies $(d_0,N_2)=1$. Then we have
\begin{equation}
	\mathcal{W}_{f_1,f_2}(b,l,q;\mathbf{x}_0,\mathbf{y}_0,r_0) \ll_{f_1,f_2,l,\mathbf{x}_0,\mathbf{y}_0,r_0,\epsilon}  s(q;b)^{\epsilon-1} + s(q;b)^{A}q^{2\theta-\frac{1}{2}+\epsilon},\label{eq:shifted_con_first_bound}
\end{equation}
where $\theta$ is an admissible exponent towards the Ramanujan-Petersson conjecture as in Theorem~\ref{th:main_prim} and $A\in \R_{0}$ is an absolute constant. 

Furthermore, we have
\begin{multline}
	\mathcal{W}_{f_1,f_2}^{\textrm{pr}}(b,l,q;\mathbf{x}_0,\mathbf{y}_0,r_0) \ll_{f_1,f_2,l,\mathbf{x}_0,\mathbf{y}_0,r_0,\epsilon} \min\left(\frac{(2^{\omega(q)})^2}{s(q;b)^{1-\epsilon}},\log\log(q)^{C+1}s(q;b)^{-\frac{1}{\log\log(q)}}\right) \\ + s(q;b)^{A}q^{2\theta-\frac{1}{2}+\epsilon}, \label{eq:shifted_con_second_bound}
\end{multline}
for some absolute constant $C>0$. The implicit constants depend polynomial on $l$, $\vert x_0'\vert$, $\vert x_0\vert$, $\vert y_0'\vert+\vert y_0'\vert^{-1}$, $\vert y_0\vert+\vert y_0\vert^{-1}$ and on the denominator of $r_0$.
\end{lemmy}

\begin{proof}
As in \cite[p. 15]{BM} we start with some reduction steps. Write $r_0=c_0/d_0$. First, we replace the additive twist $n_2\mapsto e(n_2r_0)$ by a linear combination of functions $n_2\mapsto \chi(n_2)$ where $\chi$ is a Dirichlet character modulo $d_0$. Thus we need to bound
\begin{equation}
	\frac{1}{q}\left\vert \sum_{\substack{ln_1\pm bn_2\equiv 0\text{ mod }q,\\ n_1,n_2\ll q^{1+\epsilon}}} \lambda_1(n_1)\lambda_2(n_2)\chi(n_2)G\left(\frac{n_1}{q},\pm\frac{n_2}{q}\right)\right\vert. \nonumber
\end{equation}
Next let $\chi^{\ast}$ denote the primitive character underlying $\chi$. Using M\"obius inversion and the Hecke-relations we further reduce the problem to estimating
\begin{equation*}
	\sum_{\substack{g\mid f\mid d_0\\ (g,N_2)=1}}\mu(f)^2\vert \lambda_2(f/g)\vert\left\vert\frac{1}{q} \sum_{\substack{ln_1\pm bfgn_2\equiv 0\text{ mod }q \\ n_1,n_2fg\ll q^{1+\epsilon}}}\lambda_1(n_1)\lambda_2(n_2)\chi^{\ast}(n_2)G\left(\frac{n_1}{q},\pm\frac{gfn_2}{q}\right)\right\vert. \nonumber
\end{equation*}

Note that, since we are assuming $(d_0,N_2)=1$, the coefficient $[\chi^{\ast}\cdot\lambda_2](n_2)$ belongs to a newform. Thus, upon modifying $G$ and replacing $f_2$ by $\chi^{\ast}\otimes f_2$, it is sufficient to consider the sum
\begin{equation}
	\frac{1}{q}\left\vert \sum_{\substack{ln_1\pm bf'n_2\equiv 0 \text{ mod }q \\ n_1,n_2\ll q^{1+\epsilon} }} \lambda_1(n_1)\lambda_2(n_2)G(\frac{n_1}{q},\frac{n_2}{q})\right\vert. \nonumber
\end{equation}
for some $f'\mid d_0^2$. Note that we are making no assumption on $(bf',q)$ at this point. The changes necessary for these reduction steps introduce at most polynomial dependence on $d_0$ and $N_2$ in the implicit constants. However, due to Lemma~\ref{lm:comparison_s} we have $s(q;b) \asymp_{d_0} s(q;\pm f'b)$ and we are considering $d_0$ as fixed.

We continue with another reduction step addressing the presence of $l$ in the congruence condition. Let $l=l'\cdot (l,bf')$ and $b' = bf'/(l,bf')$. Since $l\mid q$ we find that a solution $(n_1,n_2)\in \Z^2$ to $ln_1\pm bf'n_1\equiv 0\text{ mod }q$ must satisfy $n_1\mid l'$. We can thus replace $n_2$ by $l'n_2$ and obtain
\begin{equation}
	\frac{1}{q}\left\vert \sum_{\substack{n_1\pm b'n_2\equiv 0 \text{ mod }q/l \\ n_1,l'n_2\ll q^{1+\epsilon} }} \lambda_1(n_1)\lambda_2(l'n_2)G(\frac{n_1}{q},\frac{l'n_2}{q})\right\vert. \nonumber
\end{equation}
Arguing as above we can modify $G$ and use the Hecke-relation to arrive at sums of the form
\begin{equation}
	\mathcal{W}_{f'',\pm}(b';q/l) = \frac{1}{q}\left\vert \sum_{\substack{n_1\pm b'f''n_2\equiv 0 \text{ mod }q/l \\ n_1,n_2\ll q^{1+\epsilon} }} \lambda_1(n_1)\lambda_2(n_2)G(\frac{n_1}{q},\frac{n_2}{q})\right\vert, \nonumber
\end{equation}
for $f''\mid l^2$. Again we note that these reduction steps introduce at most polynomial dependence on $l$ and that
\begin{equation}
	s(q/l;f''b') \asymp_{d_0,l} s(q;b). \nonumber
\end{equation}
To simplify notation we put $q'=q/l$ and we keep in mind that $q\asymp_l q'$.

Let $(x_1,x_2)^{\top}\in \Lambda_{q',\pm f''b'}$ be a minimal vector and write $r=(x_1,x_2)$. Further set $x_1'=x_1/r$ and $x_2'=x_2/r$ and choose a factorization $r=r_1r_2$ with $(r_i,x_i')=1$ and $(r_1,r_2)=1$. After recalling Lemma~\ref{lm:prop_reduced_bas} we can write
\begin{equation}
	\mathcal{W}_{f'',\pm}(b';q') \leq \sum_{\frac{q'}{r}\mid h}\Bigg\vert \frac{1}{q} \sum_{\substack{n_1x_2'-n_2x_1'=h \\ n_ix_i'\equiv hy_i\text{ mod }r_i \\ n_1,n_2\ll q^{1+\epsilon} }} \lambda_1(n_1)\lambda_2(n_2)G(\frac{n_1}{q},\frac{n_2}{q})\Bigg\vert. \nonumber
\end{equation}

We first handle the diagonal contribution, denoted by $\Delta$, coming from $h=0$. Note that since $(x_1',x_2')=1$ the condition $n_1x_2'-n_2x_1'=0$ together with the congruences $n_i\equiv 0 \text{ mod }r_i$ implies that $x_1\mid n_1$ and $x_2\mid n_2$. Thus we can write 
\begin{equation}
	\Delta = \frac{1}{q} \left\vert \sum_{n\ll \frac{q^{1+\epsilon}}{s(b,q)}} \lambda_1(x_1n)\lambda_2(x_2n)G\left(\frac{x_1n}{q},\frac{x_2n}{q}\right) \right\vert. \nonumber
\end{equation} 
Using Rankin-Selberg theory we can estimate this by
\begin{equation}
	\Delta \ll s(b,q)^{\epsilon -1}. \nonumber
\end{equation}
Let us denote the remaining part of $\mathcal{W}_{f'',\pm}(b';q')$ by $\mathcal{W}_{f'',\pm}^{\ast}(b';q')$, so that
\begin{equation}
\mathcal{W}_{f'',\pm}(b';q') \ll \mathcal{W}_{f'',\pm}^{\ast}(b';q') + s(b,q)^{\epsilon -1}. \nonumber
\end{equation}
In $\mathcal{W}_{f'',\pm}^{\ast}(b';q')$ we detect the congruence condition on $n_1$ (resp. $n_2$) using additive characters. In view of the triangle inequality we obtain that
\begin{equation}
	\mathcal{W}_{f'',\pm}(b';q') \leq \max_{\substack{a_1\text{ mod } r_1,\\ a_2\text{ mod }r_2}}\sum_{\substack{\frac{q'}{r}\mid h,\\ h\neq 0}}\Bigg\vert \frac{1}{q} \sum_{\substack{n_1x_2'-n_2x_1'=h \\ n_1,n_2\ll q^{1+\epsilon} }} e(\frac{a_1n_1}{r_1})\lambda_1(n_1)e(\frac{a_2n_2}{r_2})\lambda_2(n_2)G(\frac{n_1}{q},\frac{n_2}{q})\Bigg\vert. \nonumber
\end{equation}
For $i=1,2$ a quick computation shows that the Maa\ss\  form with Fourier coefficients  $e(\frac{a_in}{r_i})\lambda_i(n)$, which we temporarily denote by $f_i^{(\frac{a_i}{r_i})}$, is automorphic for the lattice
\begin{equation}
	\widetilde{\Gamma}_i = \left(\begin{matrix} r_i & 0 \\ 0 & 1 \end{matrix}\right)\Gamma_1(r_i^2N_i)\left(\begin{matrix}r_i^{-1} & 0 \\ 0 & 1\end{matrix}\right). \nonumber
\end{equation}
We also check that
\begin{equation}
	\frac{1}{\Vol(\widetilde{\Gamma_i}\backslash \Hb)}\int_{\widetilde{\Gamma_i}\backslash \Hb}\vert f_i^{(\frac{a_i}{r_i})}(z)\vert^2d\mu(z)=r_i^{2+o(1)}.\nonumber
\end{equation}
After expanding $f_i^{(\frac{a_i}{r_i})}$ into a suitable basis of old and newforms, as for example given in \cite[Lemma~2]{BM14}, we find that 
\begin{equation}
	f_i^{(\frac{a_i}{r_i})}(z) = \sum_{\tilde{N}_i\mid r_i^2N_i}\sum_{d_i\mid \frac{r_i^2N_i}{\tilde{N}_i}}\sum_{g_i\text{ of level }\tilde{N}_i} a_{g_i,d_i}\cdot g_i(\frac{d_i}{r_i}z).\nonumber
\end{equation}
Note that this decomposition is not orthogonal, but we still have $\vert a_{g_i,d_i} \vert \ll_{N_i,f_i} r_i^3.$ Write $\frac{d_i}{r_i}=\frac{d_i'}{r_i'}$ in lowest terms. Thus, after possibly loosing a polynomial factor of $r$ it is sufficient to bound
\begin{equation}
	\sum_{\substack{\frac{q'}{r}\mid h,\\ h\neq 0}}\Bigg\vert \frac{1}{q} \sum_{\substack{n_1d_1'x_2'-n_2d_2'x_1'=h \\ d_1'n_1,d_2'n_2\ll q^{1+\epsilon} }} \widetilde{\lambda}_1(r_1'n_1)\widetilde{\lambda}_2(r_2'n_2)G(\frac{d_1'n_1}{q},\frac{d_2'n_2}{q})\Bigg\vert, \nonumber
\end{equation}
where $\widetilde{\lambda}_1$ (resp. $\widetilde{\lambda}_2$) is the Fourier coefficient of a $L^2$-normalized Hecke-Maa\ss\  newform $g_1$ (resp. $g_2$)of level $\tilde{N}_1\mid r_1^2N_1$ (resp. $\tilde{N}_2\mid r_2^2N_2$) and some nebentypus. The factors $r_1'$ and $r_2'$ in the Fourier coefficients can be removed using the Hecke relations as above. We thus have to bound
\begin{equation}
	\widetilde{\mathcal{W}} = \sum_{\substack{\frac{q'}{r}\mid h,\\ h\neq 0}}\Bigg\vert \frac{1}{q} \sum_{\substack{n_1x_2''-n_2x_1''=h \\ n_i\ll q^{1+\epsilon}/k_i}} \widetilde{\lambda}_1(n_1)\widetilde{\lambda}_2(n_2)\widetilde{G}(\frac{n_1}{q},\frac{n_2}{q})\Bigg\vert, \nonumber
\end{equation}
Let us stress again that these reduction steps only introduce polynomial dependence in $N_1,N_1$ and $r$. In particular, the new quantities $x_1'',x_2'',k_i$ depend at most polynomial on $r$. Similarly, $\widetilde{G}$ is obtained from $G$ by a mild dilation.

We can finally estimate $\widetilde{W}$ using the results from Appendix~\ref{sec:app}. To do so we apply smooth partitions of unity to the $n_1,n_2$ and $h$-sums. Then we can apply Proposition~\ref{pr:shifted_conv} with
\begin{equation}
	d=q'/r,\, M_1,M_2 \ll q^{1+\epsilon},\, H\ll l_1M_1+l_2M_2 \text{ and }l_1,l_2\ll r^{A_1}\cdot s(q;b), \nonumber
\end{equation}
for some effective constant $A_1\geq 0$. We obtain 
\begin{equation}
	\mathcal{W}_{f'',\pm}^{\ast}(b';q') \ll r^{A_2}s(q;b)^{2-\theta+\epsilon}q^{2\theta-\frac{1}{2}+\epsilon}.\label{eq:compare4.7}
\end{equation}
Here $A_2\geq 0$ is another in effective constant, which arises from the polynomial dependence of the implicit constant in Proposition~\ref{pr:shifted_conv} on $\tilde{N}_1$, $\tilde{N}_2$ and on $\widetilde{G}$, all of which can depend polynomially on $r_1$ and $r_2$. After collecting all the dependencies on $r$ and $s(q;b)$ together we arrive at \eqref{eq:shifted_con_first_bound}.

We turn towards estimating $\mathcal{W}_{f_1,f_2}^{\textrm{pr}}(b,l,q;\mathbf{x}_0,\mathbf{y}_0,r_0)$. Starting from \eqref{eq:original_pr_Weyl_sum} we can perform the same reduction steps as above. It turns out that we have to bound
\begin{equation}
	 \widetilde{W}_{f'',\pm}(b';q') =  \sum_{l_1\mid l}\sum_{\substack{l_1\mid d\mid q,\\ (d,l/l_1)=1}} \frac{d}{\varphi(q)}\mu(q/d)^2\left\vert\frac{1}{q} \sum_{\substack{l'n_1\pm b'f''n_2\equiv 0 \text{ mod }d' \\ n_1,n_2\ll q^{1+\epsilon} }} \lambda_1(n_1)\lambda_2(n_2)G(\frac{n_1}{q},\frac{n_2}{q})\right\vert. \nonumber
\end{equation}
where as above $d'=d/l_1$ and $b'$ is of the form $bf'/(l_1,bf')$ for some $f'\mid d_0^2$. We also have set $l'=l/l_1$. Proceeding as above we replace the congruence $l'n_1\pm bf'n_2\equiv 0 \text{ mod }d'$ by $n_1x_2^{(d')}-n_2x_1^{(d')} \equiv 0 \text{ mod }d'$ as well as additional congruences arising from the possible $\textrm{gcd}$ of $x_1^{(d')}$ and $x_2^{(d')}$, where $(x_1^{(d')},x_2^{(d')})^{\top}$ is a minimal vector in the lattice $$\Lambda(d') = \{n_1,n_2\in \Z^2\colon l'n_1+b'f'' n_2\equiv 0\text{ mod }d'\}.$$  Note that $s(\Lambda(d'))\asymp_{l,d_0} s(d;b)$, so that we can continue as earlier. We treat the diagonal contribution using the Rankin-Selberg method and the off-diagonal using Proposition~\ref{pr:shifted_conv}. We arrive at the bound
\begin{align}
	\widetilde{W}_{f'',\pm}(b';q') &\ll \sum_{d\mid q} \frac{d}{\varphi(q)}\mu(q/d)^2 \left(s(b,d)^{\epsilon-1}+q^{\theta+\epsilon}s(d;b)^{2-\theta}d^{\theta-\frac{1}{2}}\right)\nonumber \\
	&\ll s(q;b)^{A}q^{2\theta-\frac{1}{2}+\epsilon} + \sum_{d\mid q}\frac{d}{\varphi(q)}\mu(q/d)^2 s(d;b)^{\epsilon-1}.\nonumber
\end{align}
We can estimate the remaining $d$-sum in two ways. First, using \eqref{eq:minimum_bound} we get
\begin{equation}
	\sum_{d\mid q}\frac{d}{\varphi(q)}\mu(q/d)^2 s(d;b)^{\epsilon-1} \ll s(q;b)^{\epsilon-} \prod_{p\mid q}\left(\frac{1+p^{\epsilon}}{1-p^{-1}}\right) \ll \frac{(2^{\omega(q)})^2}{s(q;b)^{1-\epsilon}}. \nonumber
\end{equation}
On the other hand we can argue as follows. We write $\varphi(q)=q\cdot \mathcal{P}_q^{-1}$ with $\mathcal{P}_q=\prod_{p\mid q}(1-p^{-1})^{-1} \ll \log\log(q)$. Choose $D=\frac{q}{s(q;b)}$ and split the $d$-sum in
\begin{equation}
	\sum_{d\mid q}\frac{d}{\varphi(q)}\mu(q/d)^2 s(d;b)^{\epsilon-1} \leq \mathcal{P}_q\sum_{\substack{d\mid q \\ d\leq D}}\frac{d}{q}\mu(q/d)^2+s(q;b)^{\epsilon-1}\mathcal{P}_q \sum_{\substack{d\mid q \\ d> D}}\left(\frac{d}{q}\right)^{\epsilon}\mu(q/d)^2. \nonumber
\end{equation}
Here we have estimated $s(b,d)\geq 1$ for $d\leq D$ and we used \eqref{eq:minimum_bound} for $d>D$. Let $\beta=\frac{1}{\log\log(q)}$ and estimate
\begin{align}
	\sum_{\substack{d\mid q \\ d\leq D}}\frac{d}{q}\mu(q/d)^2 &= \sum_{\substack{d\mid q \\ d\geq  q/D}}d^{-1}\mu(d)^{2} \leq \left(\frac{D}{q}\right)^{\beta}\sum_{d\mid q}\mu(d)^2d^{\beta-1} \nonumber \\
	&= \left(\frac{D}{q}\right)^{\beta} \prod_{p\mid q}(1+p^{\beta-1}) \leq \left(\frac{D}{q}\right)^{\beta} \exp\left(C\cdot \log(q)^{\beta}\log\log\log(q)\right)\nonumber\\
	&= \left(\frac{D}{q}\right)^{\frac{1}{\log\log(q)}}\cdot \log\log(q)^C, \nonumber
\end{align}
for some positive constant $C$. Similarly, choosing $1-\alpha-\epsilon=\frac{1}{\log\log(q)}$, we obtain
\begin{equation}
	\sum_{\substack{d\mid q \\ d> D}}\left(\frac{d}{q}\right)^{\epsilon}\mu(q/d)^2 \leq \left(\frac{q}{D}\right)^{\alpha} \prod_{p\mid q}(1+p^{-\alpha-\epsilon}) \leq \left(\frac{q}{D}\right)^{1-\epsilon-\frac{1}{\log\log(q)}}\cdot \log\log(q)^C. \nonumber
\end{equation}
Combining everything we find that
\begin{equation}
	\sum_{d\mid q}\frac{d}{\varphi(q)}\mu(q/d)^2 s(d;b)^{\epsilon-1}\ll  \log\log(q)^{C+1}s(q;b)^{-\frac{1}{\log\log(q)}}. \nonumber
\end{equation}
This shows \eqref{eq:shifted_con_second_bound} and thus completes the proof.
\end{proof}

\begin{rem}
Note that our bound for $\mathcal{W}_{f_1,f_2}(b,l,q;\mathbf{x}_0,\mathbf{y}_0,r_0)$ given in \eqref{eq:compare4.7} is weaker than the corresponding bound in \cite[(4.7)]{Bl} and we have made no attempt to explicate the exponent $A$. One reason for this is that, for general $q$, the condition $(h/q,q)=1$ is not necessarily satisfied. Dropping this condition produces several technical issues in the treatment of the spectral convolution problem, which then lead to a slightly weaker bound. We also want to point out that, if the lattice $\Lambda_{q',\pm f''b'}$, for  $q'$, $f''$ and $b'$ as in the proof above, has a minimal vector $(x_1,x_2)^{\top}$ with co-prime coordinates, then we can take $A=2-\theta$.
\end{rem}

\begin{rem}
The last remaining condition $(q,N_2)=1$ appearing in Lemma~\ref{lm:shifted_conv_est} is only for technical convenience. It should be possible to remove this condition working along the lines of \cite{DHML}.  However, in our application of Lemma~\ref{lm:shifted_conv_est} this condition will always be satisfied.
\end{rem}	
	
The key result of this section, which is a generalization of \cite[Theorem~4.1]{BM}.

\begin{prop}\label{pr:normal_weyl_sum_est}
Let $q\in \N$ be large, $l\mid q$, $b\in \Z/q\Z$ and let $s=s(q;b)$. Further let $f_1,f_2$ be two $L^2$-normalized cuspidal Hecke-Maa\ss\  newforms of level $N_1,N_2$.  Assume that the Ramanujan-Petersson conjecture holds for $f_1$ and $f_2$ at all places. Let $x_0\in \R$, $y_0\in \mathbb{R}^{\times}$ and $r_0\in \Q$. Suppose that the denominator of $r_0$ is coprime to $N_2$. Then we have
\begin{equation}
	\mathcal{W}_{f_1,f_2}(b,l,q;\mathbf{x}_0,\mathbf{y}_0,r_0) \ll_{f_1,f_2,l,\mathbf{x}_0,\mathbf{y}_0,r_0,\epsilon} s^{\epsilon-1}+ \log(q)^{\epsilon-\frac{1}{9}}, \nonumber
\end{equation}
for any $\epsilon>0$. The implied constant depends at most polynomially on $l$, $\vert x_0'\vert$, $\vert x_0\vert$, $\vert y_0'\vert+\vert y_0'\vert^{-1}$, $\vert y_0\vert +\vert y_0\vert^{-1}$, the denominator of $r_0$ as well as on the conductors of $f_1$ and $f_2$.
\end{prop}	
\begin{proof}
This follows directly after combining Lemma~\ref{lm:sieve_est} and Lemma~\ref{lm:shifted_conv_est}.
\end{proof}

We also obtain the following result for the Weyl sums arising in the case of primitive points.
\begin{prop}\label{pr:primitive_weyl_sum_est}
Let $q\in \N$ be large, $l\mid q$ and $b\in \Z/q\Z$. Further let $f_1,f_2$ be two $L^2$-normalized cuspidal Hecke-Maa\ss\  newforms of level $N_1,N_2$.  Assume that the Ramanujan-Petersson conjecture holds for $f_1$ and $f_2$ at all places. Let $x_0\in \R$, $y_0\in \mathbb{R}^{\times}$ and $r_0\in \Q$. Suppose that the denominator of $r_0$ is coprime to $N_2$. 
Finally, fix $\epsilon>0$ and assume that (at least) one of the following two conditions is satisfied:
\begin{enumerate}
	\item $2^{\omega(q)}\ll \log(q)^{\frac{1}{10}}$ and $q^{\epsilon}\ll s(q;b)$;
	\item $\min(2^{(2+\epsilon)\omega(q)}, \exp((C+2)\log\log(q)\log\log\log(q))) \ll s(q;b) \ll q^{\frac{1/2-2\theta}{A}-\epsilon}$, for an absolute constant $C>0$, $A$ as in Lemma~\ref{lm:shifted_conv_est} and $\theta\leq \frac{7}{64}$ the currently best known bound towards the Ramanujan-Petersson conjecture.
\end{enumerate}
Then we have
\begin{equation}
	\mathcal{W}_{f_1,f_2}^{\textrm{pr}}(b,l,q;\mathbf{x}_0,\mathbf{y}_0,r_0) = o_{f_1,f_2,l,\mathbf{x}_0,\mathbf{y}_0,r_0,\epsilon}(1). \nonumber
\end{equation}
The implied constant depends at most polynomially on $l$, $\vert x_0'\vert$, $\vert x_0\vert$, $\vert y_0'\vert+\vert y_0'\vert^{-1}$, $\vert y_0\vert +\vert y_0\vert^{-1}$, the denominator of $r_0$ and on the conductors of $f_1$ and $f_2$.
\end{prop}	
\begin{proof}
If we are in the first situation, then the result follows from Lemma~\ref{lm:sieve_est}. Otherwise, if we are in the second application, then we can apply Lemma~\ref{lm:shifted_conv_est}.
\end{proof}

\begin{rem}
Unfortunately the statement of Proposition~\ref{pr:primitive_weyl_sum_est} contains some technical restrictions on the size of $s(q;b)$, which are the origin for the restrictions imposed on $s(q_n;b_n)$ in Theorem~\ref{th:main_prim}. We have not tried very hard to optimize these dependencies, but it seems unlikely that our methods are sufficient to remove these constraints completely.
\end{rem}

\section{The locally shallow case}

Throughout this section we fix two primes $l_1,l_2$. We will then consider a sequence of pairs $((q_n,b_n))_{n\in \N}$ satisfying $(q_n,l_1l_2)=1$ for all $n\in \N$ and $s(q_n;b_n)\to \infty$. Our goal is to prove Theorem~\ref{th_shallow} stated in the introduction, so that at some point we will specialize to $l_1=37$ and $l_2=53$. This will be done by closely following the argument from \cite{BM}.

\subsection{Measure classification}\label{measure_class}

In this section we will apply the joinings theorem of Einsiedler and Lindenstrauss to obtain structural information on the weak-star limits of the measures in question. To do so we put $D=l_1l_2$ and recall the corresponding finite set of places $S_D$ defined in \eqref{eq:def_SD}. By the assumption made in the beginning of this section we have $(q_n,D)=1$ for all $n$, so that Lemma~\ref{lm:co_prime_shape_sets} applies for all $n$. Thus, passing to the product we obtain
\begin{equation}
	H_{q_n,b_n}^{S_D,N} = \{ (x_{q_n,a}^{S_D},x_{q_n,b_na}^{S_D})\colon a\in \Z/q_n\Z \} \subseteq X(N)^{S_D}\times X(N)^{S_D}. \nonumber
\end{equation}
We equip $H_{q_n,b_n}^{S_D,N}$ with the normalized counting measure $\mu_{q_n,b_n}^{(S_D,N)}$. Let $\sigma$ be a weak-star limit point of the sequence $(\mu_{q_n,b_n}^{(S_D,N)})_{n\in \N}$. After passing to a subsequence if necessary we can assume that 
\begin{equation}
	\mu_{q_n,b_n}^{(S_D,N)} \to \sigma. \nonumber
\end{equation}
We recall the following basic properties of $\sigma$.

\begin{lemmy}\label{lm:basic_prop}
Let $\textrm{pr}_1$ (resp. $\textrm{pr}_2$) denote the projection $X(N)^{S_D}\times X(N)^{S_D}\to X(N)^{S_D}$ on the first (resp. second)	factor. Then the image $\sigma_i$ of $\sigma$ under $\textrm{pr}_i$ is the Haar probability measure on $X(N)^{S_D}$. Furthermore, the sets $H_{q_n,b_n}^{S_D,N}$ are invariant under right multiplication by
\begin{equation}
	t_i^{\triangle}=(t_i,t_i) \text{ where } t_i=\left(\left(\begin{matrix} l_1^{-1} & 0\\0&l_i \end{matrix}\right),\ldots, \left(\begin{matrix} l_1^{-1} & 0\\0&l_i \end{matrix}\right)\right)\in G(\Q_S) \nonumber
\end{equation}
for $i=1,2$. In particular, $\sigma$ is invariant under the action of $t_1^{\triangle},t_2^{\triangle}$.
\end{lemmy}
This summarizes the statements of \cite[Lemma~5.1 and~5.2]{BM}. 
\begin{proof}
The first part of the lemma is a direct consequence of Theorem~\ref{th:S-arithmetic_horocycle}. The second part follows by a simple computation as in the proof of \cite[Lemma~5.2]{BM}. Here we crucially use that $(D,q_n)=1$ for all $n\in \N$.
\end{proof}

We now define the following measures
\begin{itemize}
	\item Let $\mu_{G\times G}^{(S_D,N)}$ denote the (image if the) product Haar measure $\mu_{G(\Q_{S_D})\times G(\Q_{S_D})}$ on $X(N)^{S_D}\times X(N)^{S_D}$. (We normalize $\mu_{G\times G}^{(S_D,N)}$ to be a probability measure.)
	\item Given $h\in G(\Q_S{S_D})$ we define $\mu_{G,h}^{(S_D,N)}$ to be the (image of the) Haar measure $\mu_{G(\Q_{S_D})}$ on the $G(\Q_{S_D})$-orbit
	\begin{equation}
		(\Gamma(N)_{S_D}\times \Gamma(N)_{S_D})G^{\triangle}(\Q_{S_D})(1,h) \subseteq  X(N)^{S_D}\times X(N)^{S_D}. \nonumber
	\end{equation}
\end{itemize}

As a result of Lemma~\ref{lm:basic_prop} the measure $\sigma$ is a joining (and thus in particular a probability measure). We can therefore apply the joinings classification of Einsiedler and Lindenstrauss given in \cite[Corollary~3.4]{EL}. See \cite{Ak} for a nice introduction and for a formulation similar to the one used below \cite[Theorem~4.4]{Ka}. As in \cite[Proposition~5.1]{BM} we obtain the following result.

\begin{prop}\label{pr:measure_class_normal}
Let $\sigma$ and $D$ be as described above. Then there is $c\in [0,1]$ and a probability measure $\lambda$ on $G(\Q_{S_D})$ such that
\begin{equation}
	\sigma = (1-c)\mu_{G\times G} + c\mu_{G^{\triangle}}, \nonumber
\end{equation}
where
\begin{equation}
	\mu_{G^{\triangle}} = \int_{G(\Q_{S_D})}\mu_{G,h} d\lambda(h). \label{eq:G_triangle}
\end{equation}
\end{prop}

\begin{rem}
Note that this result does not require any assumption in $s(b_n,q_n)$. It only requires the assumption $(D,q_n)=1$, which is crucial to obtain the required invariance properties. See Lemma~\ref{lm:basic_prop}.	
\end{rem}

Essentially the same discussion applies to the sequence of measures obtained when restricting to primitive points. Indeed we consider
\begin{equation}
	\widetilde{H}_{q_n,b_n}^{S_D,N} = \{ (x_{q_n,a}^{S_D},x_{q_n,b_na}^{S_D})\colon a\in (\Z/q_n\Z)^{\times} \} \subseteq X(N)^{S_D}\times X(N)^{S_D}. \nonumber
\end{equation}
and the corresponding sequence of (probability) measures $\tilde{\mu}_{q_n,b_n}^{(S_D,N)}$. 

\begin{prop}\label{pr:measure_class_primitive}
Let $D=l_1l_2$ such that $(D,q_n)=1$ for all $n\in \N$ and let $N$ be some integer. Further let $\tilde{\sigma}$ be a weak-star limit point of the sequence $\tilde{\mu}_{q_n,b_n}^{(S_D,N)}$. Then there is $\tilde{c}\in [0,1]$ and a probability measure $\tilde{\lambda}$ on $G(\Q_{S_D})$ such that
\begin{equation}
	\tilde{\sigma} = (1-\tilde{c})\mu_{G\times G} + \tilde{c}\tilde{\mu}_{G^{\triangle}}, \nonumber
\end{equation}
where
\begin{equation}
	\tilde{\mu}_{G^{\triangle}} = \int_{G(\Q_{S_D})}\mu_{G,h} d\tilde{\lambda}(h). \nonumber
\end{equation}
\end{prop}
\begin{proof}
In view of Theorem~\ref{th:S-arithmetic_horocycle_prim} it is easy to see that $\tilde{\sigma}$ is a joining. The results follows after applying \cite[Corollary~3.4]{EL} as described in \cite{BM}.
\end{proof}

\subsection{Testing the measure}

We start by dividing the argument from \cite[Section~5.4]{BM} into certain steps that we formulate as separate lemmas. We start by making the following temporary definition.

\begin{defn}
A cuspidal automorphic representation $\pi=\otimes_v \pi_v$ of $G(\A)$ will be called \textit{globally tempered and $S$-super unramified} if it satisfies the following
\begin{enumerate}
	\item The conductor $N_{\pi}$ of $\pi$ satisfies $(N_{\pi},D)=1$ and $\pi_{\infty}$ is spherical;
	\item The local representations $\pi_v$ are tempered at all places $v$ of $\Q$; and
	\item For $v\in S$ the Langlands parameters of $\pi_v$ are trivial. 
\end{enumerate}
\end{defn}

The existence of such representations is not obvious, so that we record the existence of such representations in the form of a lemma.

\begin{lemmy}\label{lm:existence}
Let $D=37\cdot 53$. Then there exists a $S_D$-super unramified representation $\pi$. 
\end{lemmy}
\begin{proof}
This follows from the construction given in \cite[Section~5.4]{BM}.
\end{proof}

\begin{rem}
We expect that the existence of globally tempered $S_D$-super unramified representation can be shown for arbitrary $D$. The construction from \cite{BM} uses the CM form associated to a non-trivial class group characters of the real quadratic field $\Q(\sqrt{229})$. A direct generalization of this construction to arbitrary $D$ leads to an entertaining question in the realm of (inverse) arithmetic statistics. More precisely, one has to find a fundamental discriminant $N\equiv 1\text{ mod }4$ such that $K=\mathbb{Q}(\sqrt{N})$ has a non-trivial class group and such that the primes dividing $D$ split completely in $K/\mathbb{Q}$ and lie below principal ideals. Since such a general statement is not required for our purposes, we do not address it here.
\end{rem}

The notion of $S_D$-super unramified representations is justified by the following positivity result, which is crucial to our argument.

\begin{lemmy}\label{lm:positivity}
Let $\pi = \otimes_v\pi_v$ be a globally tempered and $S$-super unramified cuspidal automorphic representation of $G(\A)$ with newform $\varphi^{\circ}\in \pi$. Then, for any probability measure $\lambda$ on $G(\Q_{S_D})$, there is $R>0$ such that for $\varphi=\varphi^{\circ}\otimes \overline{(\lambda_R\ast\varphi^{\circ})}$ with $\lambda_R$ the restriction of $\lambda$ to the ball $B_R$ of radius $R$ we have that $\mu_{G\times G}(\varphi)=0$ and
\begin{equation}
	\mu_{G^{\triangle}}(\varphi)>0 \text{ for }\mu_{G^{\triangle}} = \int_{G(\Q_{S_D})}\mu_{G,h} d\lambda(h). \nonumber
\end{equation}
Note that here $\mu_{G^{\triangle}}$ is a measure on $X(N_{\pi})^{S_D}\times X(N_{\pi})^{S_D}$.
\end{lemmy}
\begin{proof}
Let $\varphi_1,\varphi_2\in \pi$ and put $\varphi=\varphi_1\otimes \overline{\varphi_2}$. First, since $\pi$ is cuspidal, the requirement $\mu_{G\times G}(\varphi)=0$ is automatic for all choices of $\varphi_1,\varphi_2$. 

Next we observe that
\begin{equation}
	\mu_{G,h}(\varphi) = \int_{X(N_{\pi})^{S_D}}\varphi_1(g)\overline{\varphi_2(gh)}dg = \langle \varphi_1,R_h\varphi_2\rangle, \nonumber
\end{equation}
where $R_g$ denotes the right regular action of $G(\Q_{S_D})$ on functions. Integrating this over $h$ with respect to $\lambda$ yields
\begin{equation}
	\mu_{G^{\triangle}}(\varphi) = \langle \varphi_1,\lambda\ast \varphi_2\rangle \text{ for }[\lambda\ast\varphi_2](g) = \int_{G(\Q_{S_D})}\varphi_2(gh)d\lambda(h).\nonumber
\end{equation}
We now let $\varphi^{\circ}\in L^2([G(\A)])$ denote that automorphic function corresponding to the (global) $L^2$-normalized new vector in $\pi$. Note that $\varphi^{\circ}$ is left $G(\Q)$ invariant. Furthermore, it is right $K(N_{\pi})^{S_D}$-invariant. We can thus view $\varphi^{\circ}$ naturally as a function on $X(N_{\pi})^{S_D}$. Since $(N_{\pi},D)=1$ and $\pi_{\infty}$ is spherical this function is $K_v$ invariant for $v\in S$. 

Now we claim that 
\begin{equation}
	\lambda\ast \varphi^{\circ} \neq 0.\nonumber
\end{equation}
This can be seen as in \cite[p. 23]{BM}. As a result we find $\delta$ satisfying
\begin{equation}
	0<\delta<\Vert\lambda\ast \varphi^{\circ}\Vert^2. \nonumber
\end{equation}
We can now find $R=R(\delta,\lambda)$ such that 
\begin{equation}
	\Vol(B_R,\lambda) \geq 1-\delta, \nonumber
\end{equation}
where $B_R$ is a ball of radius $R$ centered at the origin. We write $\lambda_R$ for the restriction $\lambda$ to $B_R$.

Finally, we define $\varphi_2=\varphi^{\circ}$ and $\varphi_1=\lambda_R\ast \varphi^{\circ}$. With these choices made we compute
\begin{equation}
	\mu_{G^{\triangle}}(\varphi) =  \langle \lambda_R\ast \varphi^{\circ}, \lambda\ast \varphi^{\circ}\rangle \geq \Vert \lambda\ast \varphi^{\circ}\Vert^2-\delta>0. \nonumber
\end{equation}
This completes the proof.
\end{proof}

\begin{lemmy}\label{lm:Weyl_sum}
Let $D\in \N$ and suppose that $(q_n,D)=1$ for all $n$. Suppose that $\pi$ is a globally tempered cuspidal automorphic representation of $G(\A)$ of conductor $(N_{\pi},D)=1$. Write $\varphi^{\circ}\colon [G(\A)]\to \C$ for the newform of $\pi$ and let $f_{\pi}\colon \Gamma_1(N_{\pi})\backslash \Hb\to \C$ denote the corresponding classical newform. Further suppose that $\sigma$ is the weak-star limit of the sequence $\mu_{q_n,b_n}^{(S_D,N_{\pi})}$.
Then, for $R> 0$ and a probability measure $\lambda$ of $G(\Q_{S_D})$, we have
\begin{equation}
	 \sigma([\lambda_R\ast \varphi^{\circ}]\otimes \overline{\varphi^{\circ}}) \ll_R \sup_{(y_0,x_0,d,r_0)\in \Omega} \left\vert \lim_{n\to\infty} \mathcal{W}_{\overline{f_{\pi}},f_{\pi}}(d\overline{b}_n,1,q_n; (0,x_0),(1,y_0), r_0) \right\vert, \nonumber
\end{equation}
where $\Omega\subseteq\R^{\times}\times \R \times \{d\mid D^{\infty}\}\times \{a/c\colon c\mid D^{\infty}\}$ is a compact set depending on  $R$. Similarly, if $\tilde{\sigma}$ is the weak-star limit of $\tilde{\mu}_{q_n,b_n}^{(S_D,N_{\pi})}$, then 
\begin{equation}
	\tilde{\sigma}([\lambda_R\ast \varphi^{\circ}]\otimes \overline{\varphi^{\circ}}) \ll_R \sup_{(y_0,x_0,d,r_0)\in \Omega} \left\vert \lim_{n\to\infty} \mathcal{W}_{\overline{f_{\pi}},f_{\pi}}^{\textrm{pr}}(d\overline{b}_n,1,q_n; (0,x_0),(1,y_0), r_0) \right\vert. \nonumber
\end{equation}
\end{lemmy}
\begin{proof}
We first observe that
\begin{equation}
		\vert\sigma([\lambda_R\ast \varphi^{\circ}]\otimes \overline{\varphi^{\circ}})\vert \leq \Vol(B_R,\lambda)\cdot \sup_{g\in G(\Q_{S_D})}\vert \sigma (R_g\varphi^{\circ}\otimes\overline{\varphi^{\circ}})\vert. \nonumber
\end{equation}
Next we recall that
\begin{align}
	\sigma (R_g\varphi^{\circ}\otimes\overline{\varphi^{\circ}}) &= \lim_{n\to\infty}\mu_{q_n,b_n}^{(S_D,N_{\pi})}(R_g\varphi^{\circ}\otimes\overline{\varphi^{\circ}}) \nonumber\\
	&= \lim_{n\to\infty}\frac{1}{q_n}\sum_{a\text{ mod } q_n}\varphi^{\circ}(g_{q_n,a}g)\overline{\varphi^{\circ}(g_{q_n,b_na})},\nonumber
\end{align}
where $g_{q,a} = u_{a/q}\cdot \iota_{\infty}(\diag(1,q))$. We first note that by de-adelizing $\varphi^{\circ}$ we get
\begin{equation}
	\varphi^{\circ}(g_{q,b_na}) = f_{\pi}\left(\frac{b_na+i}{q} \right).\label{eq:sec_fac}
\end{equation}
On the other hand we can write $g\in G(\Q_{S_D})$ as $g=(g_v)_{v\in S_D}$ with
\begin{equation}
	g_{\infty} = \left(\begin{matrix} y_0' & x_0' \\ 0 & 1\end{matrix}\right)k_{\infty} \text{ and }g_p = \left(\begin{matrix} p^{r_p} & \xi_p \\ 0 & 1\end{matrix}\right)k_{p} \nonumber
\end{equation}
for $k_v\in K_v$. Here $y_0'\in \R^{\times}$, $x_0'\in \R$, $r_p\in \Z$ and $\xi_p\in \Z[p^{-1}]$.  We put $d_g=\prod_{p\mid D} p^{\max(r_p,0)}$ and make the change $a\to ad_g\overline{d_g}$. Since $d_q\overline{d_q}\equiv 1 \text{ mod }q$ this has no effect on the second factor featuring \eqref{eq:sec_fac}. On the other hand, since $(q_n,D)=1$, we have
\begin{equation}
	u_{ad_q\overline{d_g}/q}g_p = \left(\begin{matrix} p^{r_p} & \xi_p \\ 0 & 1 \end{matrix}\right)\underbrace{u_{ad_q\overline{d_g}p^{-r_p}/q} k_p}_{\in K_p}. \nonumber
\end{equation}
Thus, if we take $Y_g=\prod_{p\mid D}p^{r_p}$ then we find that
\begin{equation}
	\Gamma(N_{\pi})^{S_D}\cdot g_{q_n,a}g\cdot \prod_{v\in S_D}K_v = \left(\begin{matrix} \frac{y_0'Y_g^{-1}}{q} & \frac{a\overline{Y_g}}{q}+\frac{x_0'}{q}+r_g\\ 0&1 \end{matrix}\right),\nonumber
\end{equation}
for some $r_g\in \Q$ with denominator dividing $D^{\infty}$. Thus we obtain
\begin{equation}
	\varphi^{\circ}(g_{q,a}g) = f_{\pi}\left(\frac{a\overline{Y_g}+Y_g^{-1}x_0+y_0'Y_g^{-1}i}{q}+r_g\right).\nonumber
\end{equation}
After changing $a\to \overline{b_n}a$ we end up with 
\begin{equation}
	\sigma (R_g\varphi^{\circ}\otimes\overline{\varphi^{\circ}}) = \lim_{n\to\infty}\frac{1}{q_n}\sum_{a\text{ mod } q_n}\overline{f_{\pi}\left(\frac{a+i}{q} \right)}f_{\pi}\left(\frac{a\overline{b_nY_g}+Y_g^{-1}x_0+y_0'Y_g^{-1}i}{q}+r_g\right),\nonumber
\end{equation}
which we recognize as the Weyl sum  defined in \eqref{eq:def_Weyl_sum}. Finally, since $B_R\subseteq G(\Q_S)$ is compact we obtain compactness of the corresponding set $\Omega$. The argument for $\tilde{\sigma}$ is essentially the same.
\end{proof}

\subsection{Proof of Theorem~\ref{th_shallow}}

We can now put things together and prove Theorem~\ref{th_shallow}. Let $D=37\cdot 53$ and assume that $((q_n,b_n))_{n\in \N}$ is a sequence of tuples with $s(q_n;b_n) \to\infty$ and $(q_n,D)=1$. Let $\pi$ be a globally tempered $S_D$-super unramified cuspidal automorphic representation. That such a representation exists was shown in Lemma~\ref{lm:existence}. Let $N=N_{\pi}$ denote the level of $\pi$ and note that by construction we have $(D,N)=1$.

Let $\sigma$ be a weak-$\ast$ limit point of the sequence $(\mu_{q_n,b_n}^{(S_D,N)})_{n\in \N}$ constructed in Section~\ref{measure_class}. After passing to a subsequence if necessary we can even assume that $\sigma=\lim_{n\to\infty}\mu_{q_n,b_n}^{(S_D,N)}$. By Proposition~\ref{pr:measure_class_normal} we find that
\begin{equation}
	\sigma = (1-c)\mu_{G\times G} + c\mu_{G^{\triangle}},\nonumber
\end{equation}
for $c\in[0,1]$ and $\mu_{G^{\triangle}}$ as in \eqref{eq:G_triangle}. 

Let us now assume that $c>0$. Then, using Lemma~\ref{lm:positivity} we find that
\begin{equation}
	\sigma(\varphi) > 0, \nonumber
\end{equation}
where $\varphi = [\lambda_R\ast \varphi^{\circ}]\otimes \overline{\varphi^{\circ}}$ is defined in terms of the new vector $\varphi^{\circ}$ of $\pi$. Let $f_{\pi}$ denote the classical newform obtained by de-adelizing $\varpi^{\circ}$. From Lemma~\ref{lm:Weyl_sum} we find that
\begin{equation}
	0 < \sigma(\varphi) \ll_R \sup_{(y_0,x_0,d,r_0)\in \Omega} \left\vert \lim_{n\to\infty} \mathcal{W}_{\overline{f_{\pi}},f_{\pi}}(d\overline{b}_n,1,q_n; (0,x_0),(1,y_0), r_0) \right\vert.\nonumber
\end{equation}
However, since $s(q_n;b_n) \to \infty$, the right hand side can be made arbitrarily small by Proposition~\ref{pr:normal_weyl_sum_est}. This is a contradiction, so that $c=0$. But, if $c=0$ we have $\sigma=\mu_{G\times G}$. Thus we have seen that
\begin{equation}
	\lim_{n\to\infty} \frac{1}{q_n} \sum_{a\text{ mod }q_n} f(x_{q_n,a}^{S_D}, x_{q_n,b_na}^{S_D}) = \int f d\mu_{G\times G} \text{ for }f\in \mathcal{C}_{c}^{\infty}(X(N)^{S(D)}\times X(N)^{S(D)}). \nonumber
\end{equation}
That this implies Theorem~\ref{th_shallow} follows directly from strong approximation.

Note that exactly the same argument works for weak star limits $\tilde{\sigma}$ of $(\tilde{\mu}_{q_n,b_n}^{(S_D,N)})_{n\in \N}.$ Indeed Proposition~\ref{pr:measure_class_primitive} tells us that
\begin{equation}
	\tilde{\sigma} = (1-\tilde{c})\mu_{G\times G} + \tilde{c}\tilde{\mu}_{G^{\triangle}}. \nonumber
\end{equation}
Testing this measure as above leads to $\tilde{c}=0$. Note that in the last step we have to apply Proposition~\ref{pr:primitive_weyl_sum_est} instead of Proposition~\ref{pr:normal_weyl_sum_est}. This leads to stronger conditions on $s(q_n;b_n)$. We end up with the following result.

\begin{theorem}\label{th_shallow_primitive}
Let $(q_n,b_n)$ be a sequence of pairs such that $s(q_n;b_n)\to\infty$ as $n\to \infty$ and let $N_0\mid (37\cdot 53)^{\infty}$. Assume  that $(q_n,37\cdot 53)=1$ for all $n\in \N$. and that, for fixed $\epsilon>0$, (at least) one of the following two conditions is satisfied:
\begin{enumerate}
	\item $2^{\omega(q_n)}\ll \log(q_n)^{\frac{1}{10}}$ and $q_n^{\epsilon}\ll s(q_n;b_n)$;
	\item $\min(2^{(2+\epsilon)\omega(q_n)}, \exp((C+2)\log\log(q_n)\log\log\log(q_n))) \ll s(q_n;b_n) \ll q_n^{\frac{1/2-2\theta}{A}-\epsilon}$, for an absolute constant $C>0$, $A$ as in Lemma~\ref{lm:shifted_conv_est} and $\theta\leq \frac{7}{64}$ the currently best known bound towards the Ramanujan-Petersson conjecture.
\end{enumerate} 
Then we have
\begin{equation}
	\lim_{n\to\infty} \frac{1}{\varphi(q_n)}\sum_{\substack{a\text{ mod } q_n\\ (a,q_n)=1}} f\left(\left(\begin{matrix}1 & a\\0&q_n\end{matrix}\right),\left(\begin{matrix}1 & b_na\\0&q_n\end{matrix}\right)\right) = \int_{\Gamma\backslash G}\int_{\Gamma\backslash G} f(g_1,g_2)dg_1dg_2, \nonumber
\end{equation} 
for all $f\in \mathcal{C}_c^{\infty}(\Gamma_1(N_0)\backslash G(\R) \times \Gamma_1(N_0)\backslash G(\R))$.
\end{theorem}

\section{The depth case} 	

Throughout this section we fix a prime $l$ and let $S=S_l = \{\infty,l\}$. Furthermore, we suppose that we have a sequence $(q_n,b_n)$ such that $s(q_n;b_n)\to \infty$ and $v_l(q_n)\to \infty$ as $n\to\infty$. For convenience we write $q_n=l^{\alpha_n}q_n'$ for $(q_n',l)=1$. Our goal is to proof Theorem~\ref{th_depth}, so that later we will specialize to $l\in \{37,53\}$.

We start by noting that in this situation we can write 
\begin{equation}
	\mu_{q_n}^{(S,N)}(f) = \frac{1}{q_n'}\sum_{a\text{ mod } q_n'}\int_{\Z_l}f\left(\iota_{\infty}\left(\begin{matrix} 1 & a \\0 & q_n'\end{matrix}\right)\iota_{l}\left(\begin{matrix} 1 & x_ll^{-\alpha_n} \\0 & l^{-\alpha_n}\end{matrix}\right)\right)dx_l.\nonumber 
\end{equation}
We define the measure
\begin{multline}
	\mu_{q_n,b_n}^{(S,N)}(f_1\otimes f_2) = \frac{1}{q_n'}\sum_{a\text{ mod } q_n'}\int_{\Z_l}f_1\left(\iota_{\infty}\left(\begin{matrix} 1 & a \\0 & q_n'\end{matrix}\right)\iota_{l}\left(\begin{matrix} 1 & x_ll^{-\alpha_n} \\0 & l^{-\alpha_n}\end{matrix}\right)\right) \\ \cdot f_2\left(\iota_{\infty}\left(\begin{matrix} 1 & b_na \\0 & q_n'\end{matrix}\right)\iota_{l}\left(\begin{matrix} 1 & b_nx_ll^{-\alpha_n} \\0 & l^{-\alpha_n}\end{matrix}\right)\right)dx_l.\nonumber 
\end{multline}
This is the correct extension of $\mu_{q_n,b_n}^{S,N}$ from Section~\ref{measure_class}. Indeed we note that if $f_1$ and $f_2$ are right-$K_l(N)$-invariant, then we obtain
\begin{equation}
	\mu_{q_n,b_n}^{(S,N)}(f_1\otimes f_2) = \frac{1}{q_n}\sum_{a\text{ mod } q_n}f_1\left(\iota_{\infty}\left(\begin{matrix} 1 & a \\0 & q_n\end{matrix}\right)\right) f_2\left(\iota_{\infty}\left(\begin{matrix} 1 & b_na \\0 & q_n\end{matrix}\right)\right)\nonumber
\end{equation}
by the Chinese Remainder Theorem. 

We now consider a weak-star limit point $\sigma$ of $\mu_{q_n,b_n}^{(S,N)}$. After passing to a subsequence if necessary we can assume that 
\begin{equation}
	\lim_{n\to\infty} \mu_{q_n,b_n}^{(S,N)}(f) = \sigma(f) \text{ for }f\in \mathcal{C}_c^{\infty}(X(N)^S\times X(N)^S). \nonumber
\end{equation}
Our goal is to show that $\sigma$ is the uniform measure. After passing to further subsequences we can make the following two assumptions:
\begin{enumerate}
	\item The sequence $\alpha_n$ is monotone increasing (i.e. $\alpha_1<\alpha_2<\ldots$). This is possible since we are assuming that $\alpha_n\to \infty$.
	\item There is $b\in \Z_l$ such that $b_n\to b$ as $n\to\infty$. We can do so because $b_n\in \Z \subseteq \Z_l$ and $\Z_l$ is compact.
\end{enumerate}
We can now see that $\sigma$ has the following invariance properties.

\begin{lemmy}
The image $\sigma_i$ of $\sigma$ under $\textrm{pr}_i$ is the Haar probability measure on $X(N)^S$. Furthermore, $\sigma$ is invariant under the action of the one-parameter unipotent group\begin{equation}
	\mathfrak{U}_b(\Q_l) = \left\{\left(\iota_l\left(\begin{matrix} 1 & z \\ 0 & 1\end{matrix}\right),\iota_l\left(\begin{matrix} 1 & bz \\ 0 & 1\end{matrix}\right)\right)\colon z\in \Q_l \right\} \subseteq G(\Q_S)\times G(\Q_S). \nonumber
\end{equation}
\end{lemmy}
\begin{proof}
The first statement is a direct consequence of Theorem~\ref{th:S-arithmetic_horocycle}. To see the invariance property we let $u=(\iota_l(u_z),\iota_l(u_{bz}))\in \mathfrak{U}_b(\Q_l)$ and take $f\in \mathcal{C}_c^{\infty}(X(N)^S\times X(N)^S)$. Pick $N=N(z)\in \N$ such that $zl^{\alpha_n}\in \Z_l$ for all $n\geq N$. For $n\geq N$ we compute
\begin{align}
	&\mu_{q_n,b_n}^{(S,N)}(R_u[f_1\otimes f_2]) \nonumber\\
	&\quad = \frac{1}{q_n'}\sum_{a\text{ mod } q_n'}\int_{\Z_l}f_1\left(\iota_{\infty}\left(\begin{matrix} 1 & a \\0 & q_n'\end{matrix}\right)\iota_{l}\left(\begin{matrix} 1 & (x_l+zl^{\alpha_n})l^{-\alpha_n} \\0 & l^{-\alpha_n}\end{matrix}\right)\right) \nonumber\\
	&\qquad\qquad \cdot f_2\left(\iota_{\infty}\left(\begin{matrix} 1 & b_na \\0 & q_n'\end{matrix}\right)\iota_{l}\left(\begin{matrix} 1 & (b_nx_l+bzl^{\alpha_n})l^{-\alpha_n} \\0 & l^{-\alpha_n}\end{matrix}\right)\right)dx_l\nonumber\\
	&\quad = \frac{1}{q_n'}\sum_{a\text{ mod } q_n'}\int_{\Z_l}f_1\left(\iota_{\infty}\left(\begin{matrix} 1 & a \\0 & q_n'\end{matrix}\right)\iota_{l}\left(\begin{matrix} 1 & x_ll^{-\alpha_n} \\0 & l^{-\alpha_n}\end{matrix}\right)\right) \nonumber\\
	&\qquad\qquad \cdot f_2\left(\iota_{\infty}\left(\begin{matrix} 1 & b_na \\0 & q_n'\end{matrix}\right)\iota_{l}\left(\begin{matrix} 1 & b_nx_ll^{-\alpha_n} \\0 & l^{-\alpha_n}\end{matrix}\right)\iota_{l}\left(\begin{matrix} 1 & (b-b_n)z \\0 & 1\end{matrix}\right)\right)dx_l\nonumber
\end{align} 
Now there is $M=M(f_2,z)\in \N$ such that 
\begin{equation}
	f_2(g\iota_l(u_{(b-b_n)z})) =  f_2(g)  \text{ for all }n\geq M. \nonumber 
\end{equation}
We conclude that for $n\geq \max(N,M)$ we have
\begin{equation}
	\mu_{q_n,b_n}^{(S,N)}(R_u[f_1\otimes f_2]) = \mu_{q_n,b_n}^{(S,N)}(f_1\otimes f_2). \nonumber
\end{equation}
This shows the desired invariance of $\sigma$ and finishes the proof.
\end{proof}

We are now ready to establish an analogous version of Proposition~\ref{pr:measure_class_normal}. However, instead of using the joinings theorem of Einsiedler and Lindenstrauss we will apply an $S$-arithmetic version of Ratner's Theorem.

\begin{prop}\label{pr:measure_class_depth}
Let $\sigma$ and $S$ be as described above. Then there is $c\in [0,1]$ and a probability measure $\lambda$ on $G(\Q_{S_D})$ such that
\begin{equation}
	\sigma = (1-c)\mu_{G\times G} + c\mu_{G^{\triangle}}, \nonumber
\end{equation}
where
\begin{equation}
	\mu_{G^{\triangle}} = \int_{G(\Q_{S_D})}\mu_{G,h} d\lambda(h). \nonumber
\end{equation}
\end{prop}
\begin{proof}
In view of the ergodic decomposition of $\sigma$ it is sufficient to study ergodic components. Thus let $\sigma'$ be an ergodic $\mathfrak{U}_b$-invariant probability measure. According to \cite[Theorem~1]{Ra} $\sigma'$ is algebraic. In particular, there is a closed subgroup $\mathfrak{U}_b\subseteq H\subseteq G(\Q_S)\times G(\Q_S)$ and $x\in X(N)^S\times X(N)^S$ such that $\sigma'$ is an $H$-invariant probability measure supported on $xH$. Since we know that $\sigma'$ the projection to each copy of $G(\Q_S)$ is the uniform measure we find that $\sigma'=\mu_{G,h}$ for some $h\in G(\Q_{S})$. This completes the argument.
\end{proof}

We can now proceed in testing the measure as above. Let $\pi$ be a globally tempered $S_l$-super unramified cuspidal automorphic representation with level $N_{\pi}$. We denote the new-vector in $\pi$ by $\varphi^{\circ}$ and let $\varphi= [\lambda_R\ast \varphi^{\circ}]\otimes \overline{\varphi^{\circ}}$. Suppose that $c>0$, so that by Lemma~\ref{lm:positivity} we have
\begin{equation}
	\sigma(\varphi) = c\cdot \mu_{G^{\triangle}}(\varphi) >0. \nonumber
\end{equation} 
We write $f_{\pi}$ for the classical newform underlying $\varphi^{\circ}$. Our goal is to extend Lemma~\ref{lm:Weyl_sum} to our situation. We start with the estimate
\begin{equation}
	\vert \sigma(\varphi)\vert \leq \Vol(B_R,\lambda)\cdot \sup_{g\in B_R} \vert \sigma (R_g\varphi^{\circ}\otimes \overline{\varphi^{\circ}})\vert.\label{eq:sup_for_depth}
\end{equation}
Inserting the definition of $\sigma$ as a limit of the measures $\mu_{q_n,b_n}^{(S_l,N_{\pi})}$ we find that
\begin{multline}
	\sigma (R_g\varphi^{\circ}\otimes \overline{\varphi^{\circ}}) = \lim_{n\to\infty}\frac{1}{q_n'}\sum_{a\text{ mod } q_n'}\int_{\Z_l}\varphi^{\circ}\left(\iota_{\infty}\left(\begin{matrix} 1 & a \\ 0 & q_n'\end{matrix}\right)\iota_l\left(\begin{matrix} 1 & x_ll^{-\alpha_n} \\ 0 & l^{-\alpha_n}\end{matrix}\right)g\right) \\
	\cdot \overline{\varphi^{\circ}\left(\iota_{\infty}\left(\begin{matrix} 1 & b_na \\ 0 & q_n'\end{matrix}\right)\iota_l\left(\begin{matrix} 1 & b_nx_ll^{-\alpha_n} \\ 0 & l^{-\alpha_n}\end{matrix}\right)\right)} dx_l. \nonumber
\end{multline}
As in the proof of Lemma~\ref{lm:Weyl_sum} we write $g=g_{\infty}g_l$ with
\begin{equation}
	g_{\infty} = \left(\begin{matrix} y_0' & x_0' \\ 0 & 1\end{matrix}\right)k_{\infty} \text{ and }g_{l} = \left(\begin{matrix} l^{t(g)} & \xi \\ 0 & 1\end{matrix}\right)k_{l}.\nonumber
\end{equation}
At this point we need to consider two cases.

First, suppose that $t(g)\leq 0$. Then we write $x_l=x_l'+l^{\alpha_n}z_l$ and split the $x_l$-integral accordingly. After de-adelizing $\overline{\varphi^{\circ}}$ we are left with
\begin{multline}
	\sigma (R_g\varphi^{\circ}\otimes \overline{\varphi^{\circ}}) = \lim_{n\to\infty}\frac{1}{q_n}\sum_{a\text{ mod } q_n}\overline{f_{\pi}\left(\frac{b_n a+i}{q_n}\right)}\int_{\Z_l}\varphi^{\circ}\left(\iota_{\infty}\left(\begin{matrix} 1 & a \\ 0 & q_n\end{matrix}\right)\iota_l\left(\begin{matrix} 1 & z_l \\ 0 & 1\end{matrix}\right)g\right)dz_l. \nonumber
\end{multline}
Further, we obtain
\begin{equation}
	\varphi^{\circ}\left(\iota_{\infty}\left(\begin{matrix} 1 & a \\ 0 & q_n\end{matrix}\right)\iota_l\left(\begin{matrix} 1 & z_l \\ 0 & 1\end{matrix}\right)g\right) = \varphi^{\circ}\left(\iota_{\infty}\left(\begin{matrix} y_0' & x_0'+a \\ 0 & q_n\end{matrix}\right)\iota_l\left(\begin{matrix} l^{t(g)} & \xi+z_l \\ 0 & 1\end{matrix}\right)\right).\nonumber
\end{equation}
In order to continue we consider two cases. First, if $r\leq 0$, we write $\xi l^{-{t(g)}} \in r_g+\Z_l$ and obtain
\begin{equation}
	\sigma (R_g\varphi^{\circ}\otimes\overline{\varphi^{\circ}}) = \lim_{n\to\infty}\frac{1}{q_n}\sum_{a\text{ mod } q_n}\overline{f_{\pi}\left(\frac{a+i}{q} \right)}f_{\pi}\left(\frac{a\overline{b_n}l^{-{t(g)}}+l^{-{t(g)}}x_0+y_0'l^{-{t(g)}}i}{q_n}+r_{g}\right)dz_l,\nonumber
\end{equation}
Note that we were able to remove the $\Z_l$-integral completely and that the denominator of $r_g$ is some power of $l$. After recalling the definition of the Weyl sum from \eqref{eq:def_Weyl_sum} we can write
\begin{equation}
	\sigma (R_g\varphi^{\circ}\otimes\overline{\varphi^{\circ}}) = \mathcal{W}_{\overline{f_{\pi}},f_{\pi}}(\overline{b_n}l^{-{t(g)}},1,q_n,(0,l^{-{t(g)}}x_0),(1,l^{-{t(g)}}y_0'),r_q).\nonumber
\end{equation}
Thus, as in the proof of Lemma~\ref{lm:Weyl_sum}, we can estimate the part of the supremum over $g$ with ${t(g)}\leq 0$ by
\begin{equation}
	\sup_{\substack{g\in B_R,\\ t(g)\leq 0}} \vert \sigma (R_g\varphi^{\circ}\otimes \overline{\varphi^{\circ}})\vert \ll_R \sup_{(y_0,x_0,d,r_0)\in \Omega_-} \left\vert \lim_{n\to\infty} \mathcal{W}_{\overline{f_{\pi}},f_{\pi}}(d\overline{b}_n,1,q_n; (0,x_0),(1,y_0), r_0) \right\vert. \nonumber
\end{equation}
for a suitably chosen set $\Omega_-$. Note that typically $d$ will not be co-prime to $q_n$. Recall that we can still apply Lemma~\ref{lm:comparison_s} to find that
\begin{equation}
	s(q_n;d\overline{b_n}) \asymp_d s(q_n;b_n) \to \infty. \nonumber
\end{equation}
Thus, after applying Proposition~\ref{pr:normal_weyl_sum_est} we find that
\begin{equation}
	\vert \sigma(\varphi)\vert \leq \Vol(B_R,\lambda)\cdot \sup_{\substack{g\in B_R\\ t(g)>0}} \vert \sigma (R_g\varphi^{\circ}\otimes \overline{\varphi^{\circ}})\vert.\nonumber
\end{equation}

We continue with the case $t(g)>0$. Here we split the $\Z_l$-integral differently. We write $x_l=x_l'+l^{\alpha_n+t(g)}z_l$. This allows us to write
\begin{multline}
	\sigma (R_g\varphi^{\circ}\otimes \overline{\varphi^{\circ}}) = \lim_{n\to\infty}\frac{1}{q_nl^{t(g)}}\sum_{a\text{ mod } q_nl^{t(g)}}\overline{f_{\pi}\left(\frac{l^{t(g)}b_n a+l^{t(g)}i}{q_nl^{t(g)}}\right)} \\
	\cdot f_{\pi}\left(\frac{a+x_0'+y_0'i}{l^{t(g)}q_n}+r_g\right). \nonumber
\end{multline}
For some $r_g\in \Q$ whose denominator is a power of $l$. We recognise this as a version of the general Weyl sum defined in \eqref{eq:def_Weyl_sum}. Indeed, we have
\begin{equation}
	\sigma (R_g\varphi^{\circ}\otimes \overline{\varphi^{\circ}}) = \lim_{n\to\infty} \mathcal{W}_{\overline{f_{\pi}},f_{\pi}}(\overline{b_n},l^{t(g)},ql^{t(g)},(0,x_0'),(1,y_0'),r_g). \nonumber
\end{equation}
This leads to
\begin{equation}
	\sup_{\substack{g\in B_R,\\ t(g)> 0}} \vert \sigma (R_g\varphi^{\circ}\otimes \overline{\varphi^{\circ}})\vert \ll_R \sup_{(y_0,x_0,d,r_0)\in \Omega_+} \left\vert \lim_{n\to\infty} \mathcal{W}_{\overline{f_{\pi}},f_{\pi}}(\overline{b_n},d,q_nd; (0,x_0),(1,y_0), r_0) \right\vert. \nonumber
\end{equation}
for a suitable set $\Omega_+$. According to Proposition~\ref{pr:normal_weyl_sum_est} the limit vanishes and we obtain the contradiction
\begin{equation}
	0>\sigma(\varphi) = 0. \nonumber
\end{equation}
We conclude that $c=0$, so that $\sigma = \mu_{G\times G}$ and we obtain the result announced in Theorem~\ref{th_depth}.

Making the obvious modifications to this argument lets us deal with the primitive case as well. One obtains the following result.

\begin{theorem}\label{th_depth_primitive}
Let $(q_n,b_n)$ be a sequence of pairs such that $s(q_n;b_n)\to\infty$ as $n\to \infty$. Suppose that there is $l\in \{37,53\}$ such that $v_l(q_n)\to\infty$ as $n\to\infty$. Fix $\epsilon>0$ and assume that (at least) one of the following two conditions is satisfied:
\begin{enumerate}
	\item $2^{\omega(q_n)}\ll \log(q_n)^{\frac{1}{10}}$ and $q_n^{\epsilon}\ll s(q_n;b_n)$;
	\item $\min(2^{(2+\epsilon)\omega(q_n)}, \exp((C+2)\log\log(q_n)\log\log\log(q_n))) \ll s(q_n;b_n) \ll q_n^{\frac{1/2-2\theta}{A}-\epsilon}$, for an absolute constant $C>0$, $A$ as in Lemma~\ref{lm:shifted_conv_est} and $\theta\leq \frac{7}{64}$ the currently best known bound towards the Ramanujan-Petersson conjecture.
\end{enumerate} 
Then we have
\begin{multline}
	\lim_{n\to\infty} \frac{1}{\varphi(q_n)}\sum_{\substack{a\text{ mod } q_n\\ (a,q_n)=1}} f\left(\frac{1}{\sqrt{q_n}}\left(\begin{matrix}1 & a\\0&q_n\end{matrix}\right),\frac{1}{\sqrt{q_n}}\left(\begin{matrix}1 & b_na\\0&q_n\end{matrix}\right)\right) \\ = \frac{9}{\pi^2}\int_{\SL_2(\Z)\backslash \Hb}\int_{\SL_2(\Z)\backslash \Hb} f(g_1,g_2)dg_1dg_2, \nonumber
\end{multline} 
for all $f\in \mathcal{C}_c^{\infty}(\SL_2(\Z)\backslash \Hb \times \SL_2(\Z)\backslash \Hb)$.
\end{theorem}

\section{The endgame}

We are now ready to prove Theorem~\ref{th:main}. Let $((q_n,b_n))_{n\in \N}$ be a sequence of integers with $(q_n,b_n)=1$ and $s(q_n;b_n)\to\infty$. We consider the measures
\begin{equation}
	\mu_{q_n,b_n}(f) = \sum_{a\text{ mod } q_n}f\left(\frac{a+i}{q_n},\frac{ab_n+i}{q_n}\right).
\end{equation}
Let $\sigma$ be a weak star limit point of the sequence $(\mu_{q_n,b_n})_{n\in \N}$. Our goal is to show that $\sigma$ is the uniform measure on $\SL_2(\Z)\backslash \Hb\times \SL_2(\Z)\backslash \Hb$. To do so we can pass to a subsequence and assume that $\mu_{q_n,b_n}$ converges to $\sigma$. We consider two cases. First, suppose there exist $l\in \{37,53\}$ and a subsequence $(n_k)_{k\in \N}$ such that $v_l(q_{n_k})\to \infty$ as $k\to \infty$. Then we can apply Theorem~\ref{th_depth}  to this subsequence and we are done.

Thus we can assume the contrary. In particular, we find that the sequences $(v_{37}(q_n))_{n\in\N}$ and $(v_{53}(q_n))_{n\in \N}$ are bounded. This means that there are only finitely many possibilities for $(q_n,(37\cdot 53)^{\infty})$. Of course one of these possibilities must occur infinitely often. Thus, after passing to a subsequence if necessary we can assume $q_n=q_0\cdot q_n'$ with $(q_n',37\cdot 53)=1$. Note that by Remark~\ref{rm:pass_to_subseq} we have $s(q_n';b_n)\to \infty$ as $n\to\infty$. After passing to another subsequence if necessary we can assume that $b_n \equiv b_0\text{ mod }q_0$ is constant as $n\in \N$ varies. 

We now use the Chinese Remainder Theorem to rewrite the measures $\mu_{q_n,b_n}$. More precisely we put $a = a_0q_n'+a_1q_0$ for $a_0 \text{ mod }q_0$ and $a_q\text{ mod }q_n'$. We define 
\begin{equation}
	g_{a_0,q_0}=\left(\begin{matrix}1 & \frac{a_0}{q_0} \\ 0& 1\end{matrix}\right) \text{ and } z_{q_0}=\frac{i}{q_0}.\nonumber 
\end{equation}
We arrive at
\begin{equation}
	\mu_{q_n,b_n}(f) = \frac{1}{q_0}\sum_{a_0=1}^{q_0}\frac{1}{q_n'}\sum_{a_1\text{ mod } q_n'}f\left(g_{a_0,q_0}\left(\begin{matrix} 1 & a_1 \\ 0 & q_n'\end{matrix}\right)z_{q_0},g_{b_0a_0,q_0}\left(\begin{matrix} 1 & b_na_1 \\ 0 & q_n'\end{matrix}\right)z_{q_0}\right).\nonumber
\end{equation}
Recall that $f$ was $\SL_2(\Z) \times \SL_2(\Z)$-invariant. We check that for $\gamma\in \Gamma_1(q_0^2)$ we have $g_{a_0,q_0}\gamma g_{a_0,q_0}^{-1}\in \textrm{SL}_2(\Z)$. Therefore, we can define the function $f_{a_0,b_0,q_0}\colon \Gamma_1(q_0^2)\backslash\SL_2(\R)\times \Gamma_1(q_0^2)\backslash\SL_2(\R) \to \C$ by
\begin{equation}
	f_{a_0,b_0,q_0}(g_1,g_2) = f\left(g_{a_0,q_0}g_1z_{q_0},g_{b_0a_0,q_0}g_2z_{q_0}\right).\nonumber
\end{equation}
This allows us to apply Theorem~\ref{th_shallow} and we obtain
\begin{align}
	&\lim_{n\to\infty} \frac{1}{q_n}\sum_{a\text{ mod } q_n} f\left(\frac{a+i}{q_n},\frac{b_na+i}{q_n}\right) \nonumber\\
	&\qquad  = \frac{1}{q_0}\sum_{a_0=1}^{q_0}\lim_{n\to\infty}\frac{1}{q_n'}\sum_{a_q\text{ mod }q_n'}f_{a_0,b_0,q_0}\left(\left(\begin{matrix} 1 & a_1\\0&q_n'\end{matrix}\right),\left(\begin{matrix} 1 & b_na_1\\0&q_n'\end{matrix}\right)\right)\nonumber \\
	&\qquad =\frac{9/\pi^2}{q_0\cdot [\SL_2(\Z)\colon \Gamma_1(q_0^2)]^2}\sum_{a_0=1}^{q_0}\int_{\Gamma_1(q_0^2)\backslash \textrm{SL}_2(\R)}\int_{\Gamma_1(q_0^2)\backslash \textrm{SL}_2(\R)}f_{a_0,b_0,q_0}(g_1,g_2)dg_1dg_2.\nonumber
\end{align}
After appropriately changing variables in the remaining integral we obtain
\begin{multline}
	\lim_{n\to\infty} \frac{1}{q_n}\sum_{a\text{ mod } q_n} f\left(\frac{a+i}{q_n},\frac{b_na+i}{q_n}\right) \\
	 = \int_{\textrm{SL}_2(\Z)\backslash\Hb}\int_{\textrm{SL}_2(\Z)\backslash \Hb}f(x_1+yiy_1,x_2+iy_2)\frac{dx_1dy_1}{y_1^2}\frac{dx_2dy_2}{y_2^2}. \nonumber
\end{multline}
This is the equidistribution statement we are looking for and the proof of Theorem~\ref{th:main} is complete.

Similarly one combines Theorem~\ref{th_shallow_primitive} and Theorem~\ref{th_depth_primitive} to establish Theorem~\ref{th:main_prim}. We omit the details.\\

\textbf{Acknowledgement:} The author would like to thank Valentin Blomer and Philippe Michel for answering questions concerning their paper \cite{BM} and for much encouragement while writing this paper. Furthermore, I would like to thank Jens Marklof for pointing me to the papers \cite{Ma1, Ma2} and Radu Toma for fruitful conversations about lattices. I thank Uri Shapira for raising many interesting questions concerning this work. One of these questions motivated us to include Theorem~\ref{th:main_prim} in this paper. Further, I would like to thank Andreas Wieser and Zuo Lin for spotting a small mistake  related to Remark~\ref{rm:Wieser}, which appeared in an earlier version of this article. Last but not least I thank the anonymous referee for many useful comments.

\appendix

\section{A shifted convolution problem}\label{sec:app}

In this appendix we will prove the following generalization of \cite[Proposition~2.1]{BM}.

\begin{prop}\label{pr:shifted_conv}
Let $N_1,N_2,d,l_1,l_2\in \N$ and $H,M_1,M_2,P_1,P_2\geq 1$. Let $f,g$ be two cuspidal newforms of levels $N_1,N_2$ respectively and central characters $\chi_f,\chi_g$ with Hecke eigenvalues $\lambda_f(m), \lambda_g(m)$. Let $G$ be a smooth function supported on $[M_1,2M_2]\times  [M_2,2M_2]$ satisfying $\Vert G^{(i,j)}\Vert_{\infty} \ll_{i,j} (P_1/M_1)^i(P_2M_2)^j$ for $i,j\in \N_0$. For $H\leq h\leq 2H$ let $\vert\alpha(h)\vert \leq 1$. Then
\begin{multline}
	\mathcal{D} = \sum_{\substack{H\leq h\leq 2H,\\ d\mid h}} \alpha(h) \left\vert \sum_{l_1m_1\pm l_2m_2 = h} \lambda_f(m_1)\lambda_g(m_2)G(m_1,m_2)\right\vert \\ 
	\ll d^{\theta-\frac{1}{2}+\epsilon}(l_1l_2)^{\frac{1}{2}-\theta+\epsilon} \cdot (l_1M_1+l_2M_2)^{\frac{1}{2}+\theta+\epsilon} H^{\frac{1}{2}+\epsilon}\left(1+\frac{H/d}{l_1l_2(1+H/l_2M_2)}\right)^{\frac{1}{2}}. \nonumber
\end{multline}
The implied constant depends on $\epsilon$ and polynomially on $P_1, P_2$ and the conductors of $f,g$.
\end{prop}

The main difference is that we have dropped the condition $(d,h/d)=1$ in the $h$-sum. This might seem like a minor detail, but it turns out that dropping this condition produced several technical complications.

The proof of Proposition~\ref{pr:shifted_conv} reduces to a suitable estimate for averages of (twisted) Kloosterman sums, which generalizes \cite[Lemma~2.1]{BM}.

\begin{lemmy}\label{lm:B2}
Let $P_0,P_1,P_2,S,H,Q\geq 1$ and $d,N\in \N$. Let $\chi$ be a Dirichlet character modulo $N$. Let $u$ be a smooth function supported on $[H,2H] \times [S,2S] \times [Q,2Q]$ satisfying $\Vert u^{(i,j,k)}\Vert_{\infty}\ll (P_0/H)^i(P_1/S)^j(P_2/Q)^k$ for $0\leq i,j,k\leq 2$. Let $a(h)$ and $b(h)$ be sequences of complex numbers supported on $H\leq h\leq 2H$ and $S\leq s \leq 2S$ with $\vert a(h)\vert,\vert b(s)\vert \leq 1$. Then
\begin{multline}
	\sum_s\sum_{N\mid q}\sum_{d\mid h} a(h)b(s)S_{\chi}(\pm h,s,q) u(h,s,q) \\ \ll d^{\theta-\frac{1}{2}+\epsilon}N^{\frac{1}{2}-\theta+\epsilon}(HS)^{\frac{1}{2}+\epsilon}Q^{1+\epsilon}\left(1+\frac{HS}{Q^2}+\frac{S}{N}\right)^{\frac{1}{2}} \left(1+\frac{H/d}{N(1+HS/Q^2)}\right)^{\frac{1}{2}}\\ \cdot \left(1+\left(\frac{HS}{Q}\right)^{-\theta}\right). \nonumber
\end{multline}
The implied constants depends on $\epsilon$ and polynomially on $P_1,P_1,P_2$.
\end{lemmy} 
As in \cite{BM} we closely follow the proof of \cite[Proposition~3.5]{Bl}. We will only briefly sketch the argument and focus on the part where different ideas become relevant.
\begin{proof}
We start by defining
\begin{equation}
	U(t_0,t_1;q) = \int_{\R} \int_{\R} u\left(x,y\frac{4\pi \sqrt{xy}}{q}\right)\cdot \frac{4\pi \sqrt{xy}}{q}e(-t_0x-t_1y) dxdy. \nonumber
\end{equation}
This function replaces $U_h(t,q)$ defined in \cite[(3.13)]{Bl} and satisfies
\begin{equation}
	\frac{\partial^n}{\partial q^n} U(t_0,t_1;q) \ll_n \left(1+\frac{Ht_0}{P_0}\right)^{-2}\left(1+\frac{St_1}{P_1}\right)^{-2}HSQ\left(\frac{QP_2}{\sqrt{HS}}\right)^n. \nonumber
\end{equation}
Fourier inversion also yields
\begin{equation}
	q\cdot u(h,s,q) = \int_{\R}\int_{\R}U\left(t_0,t_1,\frac{4\pi \sqrt{hs}}{q}\right) e(t_0h+t_1s)dt_0dt_1. \nonumber
\end{equation}
Instead of \cite[(3.16)]{Bl} we arrive at
\begin{multline}
	\sum_s\sum_{N\mid q}\sum_{d\mid h} a(h)b(s)S_{\chi}(\pm h,s,q) u(h,s,q) \\
	= \int_{\R}\int_{\R}	\sum_s b_{t_1}(s)\sum_{d\mid h} a_{t_0}(h)\sum_{N\mid q}\frac{1}{q} S_{\chi}(\pm h,s,q) U\left(t_0,t_1,\frac{4\pi \sqrt{hs}}{q}\right) dt_0dt_1, \nonumber
\end{multline}
for $a_{t_0}(h) = e(t_0h)a(h)$ and $b_{t_1}(s)= e(t_1s)b(s)$. 

At this stage we apply the Kuznetsov formula to the $q$-sum. See for example \cite[Proposition~2.3]{Bl} for a precise statement. This step is slightly different depending on our choice sign $\pm$. We will exclusively discuss the case when $h$ and $s$ have the same sign. The other case is handled similarly.

We obtain
\begin{multline}
	\sum_{N\mid q}\frac{1}{q} S_{\chi}(\pm h,s,q) U\left(t_0,t_1,\frac{4\pi \sqrt{hs}}{q}\right) \\
	= \mathcal{M}_{\textrm{Ma}}(U(t_0,t_1,\cdot);h,s)+\mathcal{M}_{\textrm{hol}}(U(t_0,t_1,\cdot);h,s)+\mathcal{M}_{\textrm{Eis}}(U(t_0,t_1,\cdot);h,s), \nonumber
\end{multline}
for 
\begin{equation}
	\mathcal{M}_{\textrm{Ma}}(U(t_0,t_1,\cdot);h,s) = \sum_{f\in \mathcal{B}(N,\chi)} \frac{\overline{\rho_f(h)}\rho_f(s)}{\cosh(\pi t_f)}\widetilde{U}(t_0,t_1;t_f) \nonumber
\end{equation}
and similar contributions of holomorphic cusp forms and Eisenstein series. In what follows we will discuss how to handle the $\mathcal{M}_{\textrm{Ma}}$ contribution. We leave it as an exercise for the reader to fill in the details necessary to handle the terms arising from $\mathcal{M}_{\textrm{hol}}$ and $\mathcal{M}_{\textrm{Eis}}$. 

After applying Cauchy-Schwarz we see that the contribution of Maa\ss\ cusp forms is bounded by
\begin{multline}
	\mathcal{S}_{\textrm{Ma}}=\int_{\R}\int_{\R}\left(\sum_{f\in \mathcal{B}(N,\chi)} \frac{\vert \widetilde{U}(t_0,t_1;t_f)\vert}{\vert \cosh(\pi t_f)\vert}\left\vert \sum_s b_{t_1}(s)\rho_f(s) \right\vert^2\right)^{\frac{1}{2}} \\ \cdot \left(\sum_{f\in \mathcal{B}(N,\chi)} \frac{\vert \widetilde{U}(t_0,t_1;t_f)\vert}{\vert \cosh(\pi t_f)\vert}\left\vert \sum_{d\mid h} \overline{a_{t_0}(h)}\rho_f(h) \right\vert^2\right)^{\frac{1}{2}} dt_0dt_1.\nonumber
\end{multline}

Now we recall the large sieve inequality
\begin{equation}
	\sum_{\substack{f\in \mathcal{B}(N,\chi),\\ \vert t_f\vert \leq T}} \frac{1}{\cosh(\pi t_f)}\left\vert \sum_{s}b_{t_1}(s)\rho_f(s)\right\vert^2 \ll_{\epsilon} S^{1+\epsilon}\left(T^2+\frac{S}{N}\right). \nonumber
\end{equation}
from \cite[Proposition~3.3]{Bl}. This is sufficient to handle the part containing the $s$-sum. Here we derive the estimate 
\begin{align}
	&\sum_{\substack{f\in \mathcal{B}(N,\chi),\\ \vert t_f\vert \leq T}} \frac{1}{\cosh(\pi t_f)}\left\vert \sum_{d\mid h}a_{t_0}(h)\rho_f(h)\right\vert^2 \\
	&\qquad \ll (dHN)^{\epsilon}\sum_{\substack{s\mid (Nd)^{\infty}\\ d\mid s \\ s\ll H}}\sum_{\substack{f\in \mathcal{B}(N,\chi),\\ \vert t_f\vert \leq T}} \frac{1}{\cosh(\pi t_f)}\left\vert \sum_{(s,h)=1}a_{t_0}(sh)\rho_f(sh)\right\vert^2 \nonumber\\
	&\qquad \ll (dHNT)^{\epsilon} \sum_{\substack{s\mid (Nd)^{\infty}\\ d\mid s \\ s\ll H}} s^{2\theta}(N,s)^{1-2\theta}\left(T^2+\frac{H}{sN}\right)\frac{H}{s} \nonumber\\
	&\qquad \ll (dHNT)^{\epsilon} N^{1-2\theta}d^{2\theta-1}H\left(T^2+\frac{H}{dN}\right) \nonumber
\end{align}
from \cite[(5.12)]{BM14}.

Further, recall that $\widetilde{U}(t_0,t_1;t_f)$ is defined in \cite[(2.19)]{Bl} and satisfies the bounds given in \cite[Lemma~2.4]{Bl} (with the obvious modifications to our setting). This allows us to complete the argument as in \cite{Bl}. 
\end{proof}

We can now complete the proof of Proposition~\ref{pr:shifted_conv}. As in \cite{BM} this is done by closely following the steps in the proof of \cite[Theorem~1.3]{Bl}. The key changes are already highlighted at the end of \cite[Section~2]{BM}. We replace \cite[Proposition~3.5]{Bl} with our Lemma~\ref{lm:B2}, which in turn is a modification of \cite[Lemma~2.1]{BM}. We thus have to replace the factor $h^{\theta}$ in \cite[(4.19)]{Bl} with $$ d^{\theta-\frac{1}{2}+\epsilon}H^{\frac{1}{2}+\epsilon}\left(1+\frac{H/d}{l_1l_2(1+H/l_2M_2)}\right)^{\frac{1}{2}}.$$  It is then straight forward to complete the proof of Proposition~\ref{pr:shifted_conv}.


\end{document}